\documentclass[12pt]{article}
\usepackage[utf8]{inputenc}
\usepackage{dcolumn,lscape}
\usepackage{amsmath,longtable,multicol,dcolumn,tabularx,graphicx,amssymb}
\usepackage{exscale,amsthm}
\usepackage{natbib}
\usepackage{tikz,dsfont}
\newcommand{\xvec}{\boldsymbol}
\newcommand{\xmat}{\mathbf}
\newcommand{\xset}{\mathds}

\newtheorem{theorem}{Theorem}
\newtheorem{example}{Example}

\newtheorem{lemma}{Lemma}

\begin{document}

\def\spacingset#1{\renewcommand{\baselinestretch}%
{#1}\small\normalsize} \spacingset{1}


\title{\bf Generalized Spatial and Spatiotemporal Autoregressive Conditional Heteroscedasticity}
  \author{Philipp Otto \footnote{Corresponding author (email: potto@europa-uni.de)}\\
    Department of Quantitative Methods, \\
    European University Viadrina, Frankfurt (Oder), Germany\\
    and \\
    Wolfgang Schmid \\
    Department of Quantitative Methods, \\
    European University Viadrina, Frankfurt (Oder), Germany\\
    and \\
    Robert Garthoff,\\
    Statistisches Landesamt des Freistaates Sachsen, Kamenz, Germany}
  \maketitle

\begin{abstract}
In this paper, we introduce a new spatial model that incorporates heteroscedastic variance depending on neighboring locations. The proposed process is regarded as the spatial equivalent to the temporal autoregressive conditional heteroscedasticity (ARCH) model. We show additionally how the introduced spatial ARCH model can be used in spatiotemporal settings. In contrast to the temporal ARCH model, in which the distribution is known given the full information set of the prior periods, the distribution is not straightforward in the spatial and spatiotemporal setting. However, it is possible to estimate the parameters of the model using the maximum-likelihood approach. Via Monte Carlo simulations, we demonstrate the performance of the estimator for a specific spatial weighting matrix. Moreover, we combine the known spatial auto\-regressive model with the spatial ARCH model assuming heteroscedastic errors. Eventually, the proposed autoregressive process is illustrated using an empirical example. Specifically, we model lung cancer mortality in 3108 U.S. counties and compare the introduced model with two benchmark approaches.
\end{abstract}

\noindent%
\vfill

\newpage
\spacingset{1.45} 

\section{Introduction}\label{sec:introduction}

Various specifications of spatial autoregressive models have been proposed in past and current literature (cf. \citealt{Anselin10}). In particular, the spatial models introduced by \cite{Whittle54} were extended to incorporate external regressors (see, e.g., \citealt{Elhorst10} for an overview), and autocorrelated residuals (e.g., \citealt{Fingleton08}), respectively. Currently, these spatial models are widely implemented in statistical software packages such that it is simple to model spatial clusters of high and low observations. Consequently, a wide range of applications can be found in empirical research, including econometrics (e.g., \citealt{Holly10a,Fingleton08b}), biometrics (e.g., \citealt{Shinkareva06,Ho05,Macnab01}) or environmetrics (e.g., \citealt{Fasso11,Fasso07,Fuentes01}).

However, spatial models that assume spatially dependent second-order moments, such as the well-known autoregressive conditional heteroscedasticity (ARCH) and generalized ARCH (GARCH) models in time series analysis proposed by \cite{Engle82} and \cite{Bollerslev86}, have not been previously discussed. \cite{Borovkova12} and \cite{Caporin06} introduced a temporal GARCH model, which includes temporal lags influenced by neighboring observations. Regarding the two-dimensional setting, \cite{Bera04} suggested a special type of a spatial ARCH model, the SARCH(1) process, that results from employing the information matrix (IM) test statistic in a simple spatial autoregressive (SAR) model. Furthermore, \cite{Noiboar05} and \cite{Noiboar07} introduced a multidimensional GARCH process to detect image anomalies. However, present extensions consider only special approaches, and no general model has been presented. Moreover, there is no strict analytical analysis of the introduced models, and
it appears that generalization of an ARCH or GARCH model to the multidimensional setting is not straightforward.

To motivate the need for a spatial ARCH model, we consider the following empirical example. In Figure \ref{fig:motivation}, the population density of all U.S. counties excluding Alaska and Hawaii ($n = 3108$) is plotted on the map. The data are from the 2010 census. Obviously, there are clusters of high population density around metropolitan areas and clusters of low population density elsewhere. This behavior can be modeled by a spatial autoregressive process; i.e., the observations are assumed to be influenced by their neighbors. The dependence between the observations can be modeled via a so-called spatial weighting matrix $\xmat{W}$. Moreover, a simple spatial autoregressive process includes an autoregressive parameter $\lambda$. Fitting the U.S. census data using this type of process leads to a model with a positive spatial correlation of $\hat{\lambda} = 0.8578$; i.e., the process identifies clusters of high and low values. This finding is not surprising. However, if we focus on the estimated residuals of the process, we observe that they are not homoscedastic but rather exhibit clusters of high and low variances, whereas the mean of the residuals is zero and not spatially autocorrelated. This means that we observe clustering behavior in the conditional spatial variances but not in the conditional means. For spatial autoregressive processes, the conditional variance is also not constant over space. However, the conditional variance of each location is independent of the variance of the surrounding locations; it depends only on the spatial weights. Thus, a new nonlinear attempt, a spatial process for conditional heteroscedasticity, is needed to achieve the required flexibility of the model. To illustrate these variance clusters, we computed for each county the sample standard deviation of the residuals lying within a radius of 500 $km$ (310.686 $mi$). In Figure \ref{fig:motivation}, the conditional sample standard deviation is visualized on the map (below, left) and by means of a simple histogram (below, right). Obviously, we observe two major clusters of the residual's variance: the variance is higher in the Eastern United States compared with the Western United States. Moreover, these two clusters are also obvious in the histogram. The estimated variance of the error process is $\hat{\sigma_\xi}^2 = 0.9104$. Certainly, the appearance of these clusters depends on the choice of the distance used for the calculation of the sample standard deviation. However, when other distances are used, we observe the same behavior.

\begin{figure}
  \centering
    \includegraphics[width=0.55\textwidth,natwidth=1000,natheight=1000, trim = 0cm 4cm 0cm 0cm, clip = true]{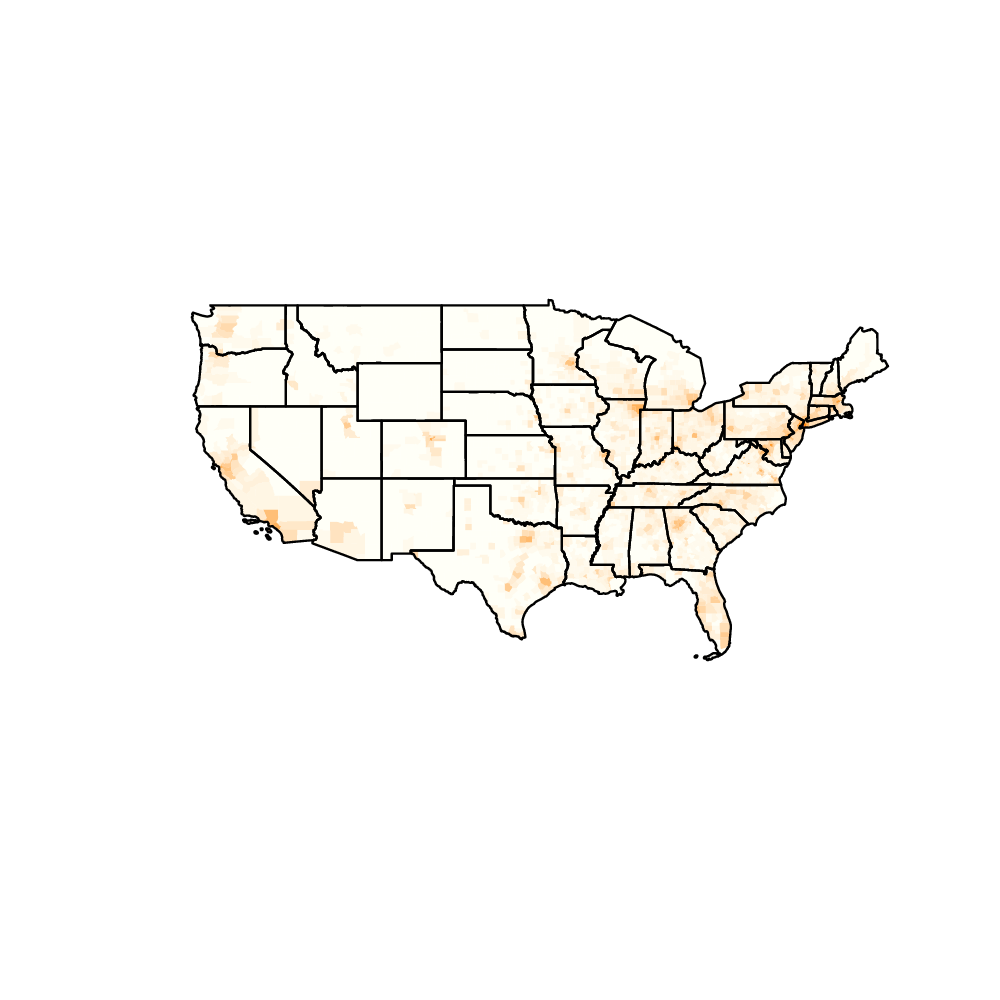}\\[-3cm]
    \includegraphics[width=0.4\textwidth,natwidth=1000,natheight=1000, trim = 0cm 4cm 0cm 0cm, clip = true]{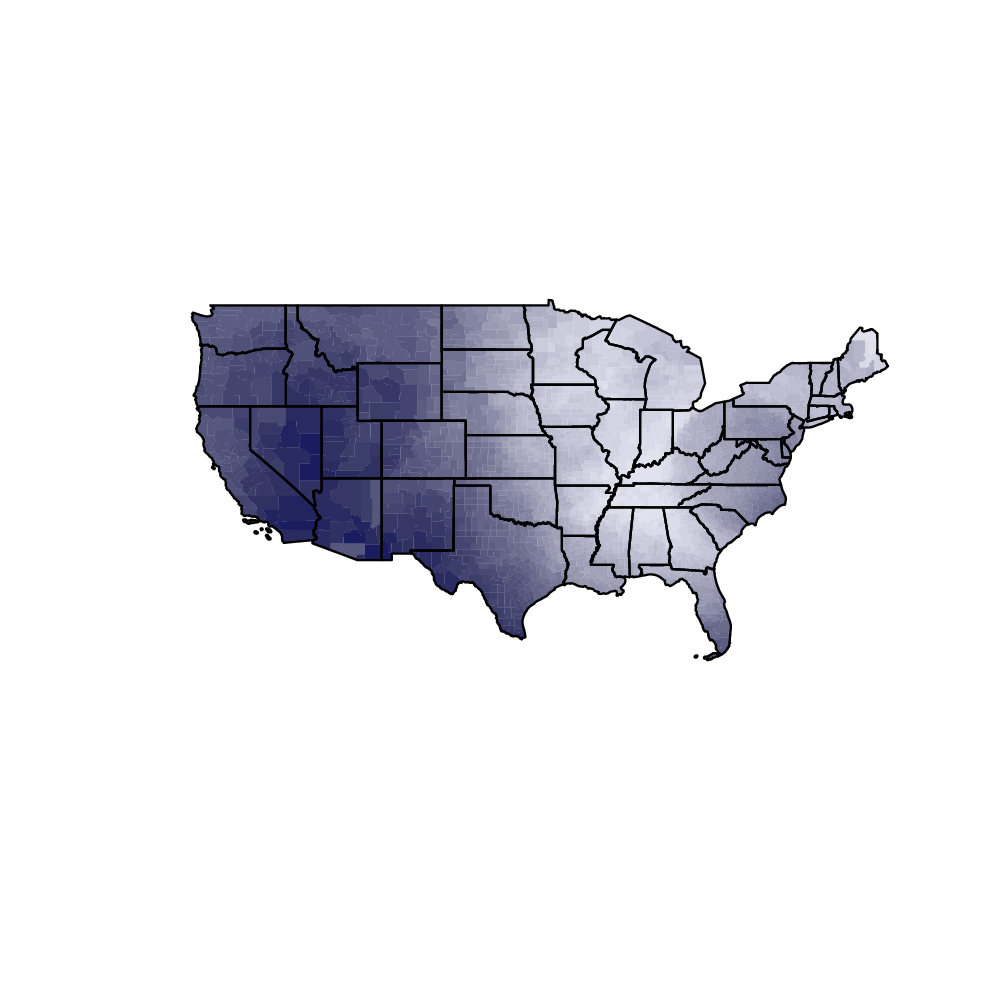}
    \includegraphics[width=0.4\textwidth,natwidth=500,natheight=500]{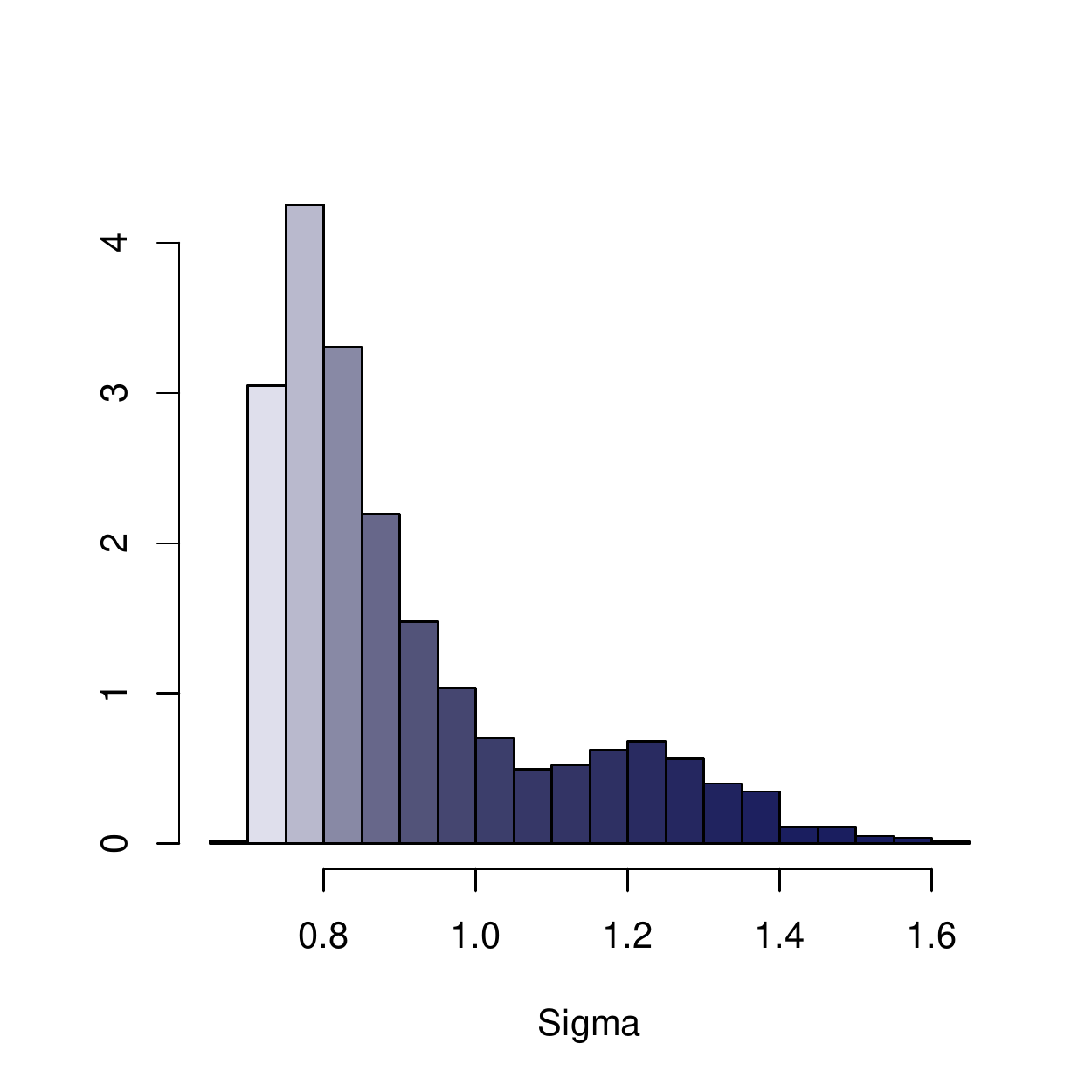}
  \caption{Population density of the U.S. counties (excluding Alaska and Hawaii) in 2010 (above). The darker the color, the higher the population density in the respective area. The sample standard deviation of the estimated residuals within a radius of 500 $km$ (310.686 $mi$) of the fitted spatial autoregressive process ($\hat\lambda = 0.8578$, $\hat{\sigma_\xi}^2 = 0.9104$) is shown below.}\label{fig:motivation}
\end{figure}

A further aspect that should be mentioned is that for spatial kriging, the underlying spatial process is usually assumed to be stationary and isotropic; i.e., the covariance between two observations depends only on the distance between these observations, not on the location of the observations (see, e.g., \citealt{Cressie93}). Thus, \cite{Sampson92}, \cite{Fuentes01}, \cite{Fuentes02}, and \cite{Schmidt03}, among others, have introduced various approaches to treat nonstationary spatial processes. For such processes, the spatial covariance matrix depends on both the location of each observation and the distances between all locations. Moreover, \cite{Ombao08}, and \cite{Stroud01} discussed nonstationary spatiotemporal models. Whereas the focus of these approaches is mostly to obtain accurate temporal forecasts or spatial interpolations, the so-called kriging, our paper aims to propose a spatial process with similar properties to the temporal ARCH process proposed by \cite{Engle82}, i.e., conditional heteroscedasticity. Hence, the entries of the spatial covariance matrix depend not only on the location and the distance between observations, as in nonstationary spatial processes, but also on the variance of locations nearby. In particular, we compare our spatial ARCH model and the temporal ARCH model with respect to important properties. Moreover, we illustrate the use of the spatial ARCH model as a residual process for spatial modeling of lung cancer mortality in the U.S. counties.

The remainder of the paper is structured as follows. In the ensuing section, we introduce the spatial ARCH model. Moreover, we derive important properties of the process. In Section \ref{sec:weighting_matrix}, two specifications of the spatial weighting matrix suitable for empirical research are discussed. Furthermore, we present some results regarding statistical inference, and we discuss an estimation procedure based on the maximum-likelihood principle. In an empirical study, we demonstrate how our results can be applied. Moreover, the results of various simulation studies are reported to yield better insight into the behavior of the spatial ARCH process. Finally, Section \ref{sec:conclusion} concludes the paper and provides some discussion of possible extensions and generalizations of the process.

\section{Spatial and Spatiotemporal Autoregressive Conditional Heteroscedasticity}\label{sec:models}

Assume that $\left\{Y(\xvec{s}) \in \xset{R}: \xvec{s} \in D_{\xvec{s}} \right\}$ is a univariate spatial stochastic process, where $D_{\xvec{s}}$ is a subset of the $q$-dimensional set of real numbers $\xset{R}^q$, the $q$-dimensional set of integers $\xset{Z}^q$, or the Cartesian product $\xset{R}^v \times \xset{Z}^l$ with $v + l = q$. Regarding the first case, a continuous process is present if a $q$-dimensional rectangle of positive volume in $D_{\xvec{s}}$ exists (cf. \citealt{Cressie11}). Considering the second case, the resulting process is a spatial lattice process. Moreover, spatiotemporal settings are covered regarding the $q$-dimensional set of integers and the product set $\xset{R}^v \times \xset{Z}^l$ because the temporal dimension can be considered as one dimension of the $q$-dimensional space. For instance, a spatiotemporal lattice process with two spatial dimensions would lie in the set of three-dimensional integers.

\subsection{Definition and Properties}\label{sec:spGARCH}

Let $\xvec{s}_1, \ldots, \xvec{s}_n$ denote all locations and $\xvec{Y}$ be the vector of observations $\left(Y\left(\xvec{s}_i\right)\right)_{i = 1, \ldots, n}$. The commonly applied spatial autoregressive model assumes that the conditional variance of $Y(\xvec{s}_i)$ depends only on the spatial weighting matrix (cf. \citealt{Cressie93,Cressie11}), not the observations at the neighboring locations. This approach is extended assuming that the conditional variance can vary over space, resulting in clusters of high and low variance.
Analogous to the ARCH time series model of \cite{Engle82}, the vector of observations is given by
\begin{equation}
\xvec{Y} = \text{diag}(\xvec{h})^{1/2} \xvec{\varepsilon} \, \label{eq:initial}
\end{equation}
where $\xvec{\varepsilon} = (\varepsilon(\xvec{s}_1), \ldots, \varepsilon(\xvec{s}_n))'$ is assumed to be an independent and identically distributed random error with $E(\xvec{\varepsilon}) = \xvec{0}$ and $Cov(\xvec{\varepsilon}) = \xmat{I}$. In addition, the identity matrix is denoted by $\xmat{I}$.
Furthermore, the vector $\xvec{h} = (h_i)_{i=1,\ldots,n}$ is specified as
\begin{equation}\label{eq:spARCH}
\xvec{h} = (h(\xvec{s}_i))_{i=1,\ldots,n} = \xvec{\alpha} + \xmat{W}  \, \text{diag}(\xvec{Y})\xvec{Y} \, ,
\end{equation}
where $\text{diag}(\xvec{a})$ denotes a diagonal matrix with the entries of $\xvec{a}$ on the diagonal.
Using the Hadamard product denoted by $\circ$, this equation can be rewritten such that
\begin{equation*}
\xvec{h} = \xvec{\alpha} + \xmat{W} (\xvec{Y} \circ \xvec{Y}) \, .
\end{equation*}
The $n \times n$ matrix $\xmat{W}$ consists of spatial weights. The elements of $\xmat{W}$ are assumed to be non-stochastic, nonnegative and zero on the main diagonal to prevent observations from influencing themselves. Moreover, each component of the vector $\xvec{\alpha} = (\alpha_i)_{i=1,\ldots,n}$ is assumed to be nonnegative.
Hence, the $i$-th entry of $\xvec{h}$ at location $\xvec{s}_i$ can be written as
\begin{equation*}
h(\xvec{s}_i) = \alpha_i + \sum_{v=1}^{n} w_{iv} Y(\xvec{s}_v)^2 \, ,
\end{equation*}
where $w_{iv}$ refers to $iv$-th entry of $\xmat{W}$ and $w_{ii}=0$ for $i=1,\ldots,n$. Thus,  $h(\xvec{s}_i)$ does not seem to depend on $Y(\xvec{s}_i)^2$. However, because $h(\xvec{s}_i)$ depends on $Y(\xvec{s}_j), j \neq i$  and these quantities depend on $Y(\xvec{s}_i)$ via $h(\xvec{s}_j)$, this is not the case. We shall discuss this point later in more detail.
Regarding this specification of $\xvec{h}$, we refer to the resulting process as the spatial ARCH model (spARCH). If $\xvec{\alpha} = \xvec{1}_n$, where $\xvec{1}_n$ is the $n$-dimensional vector of ones, and $\xmat{W} = \xmat{0}$, the resulting process coincides with the spatial white noise process.

The abovementioned spatiotemporal process could be modeled by defining the locations $\xvec{s} = (\xvec{s}_s, t)'$, where $\xvec{s}_s$ is the spatial location and $t \in \xset{Z}$ represents the point of time. For spatiotemporal settings, one must assume additionally that the weights of the locations $(\xvec{s}_s,t)$ and $(\xvec{s}_{\tilde{s}},\tilde{t})$ are zero if $\tilde{t} \geq t$. In the following Section \ref{sec:weighting_matrix}, we demonstrate how the weighting matrix must be defined for several temporal and spatiotemporal settings that have been proposed in the literature. To express the model in a more convenient manner, the time point $t$ can also be written as an index. For that reason, the number of included temporal lags is denoted by $p$, and the set of all spatial locations is $\{\xvec{s}_1, \ldots, \xvec{s}_n\}$. Thus, the process can be specified as
\begin{eqnarray*}
Y_t(\xvec{s}_i) & = & \sqrt{h_t(\xvec{s}_i)} \, \varepsilon_t(\xvec{s}_i) \qquad \text{and} \\
h_t(\xvec{s}_i) & = & \alpha_{i} + \sum_{v = 1}^{n}\sum_{\tau = 0}^{p} w_{\tau,iv} Y_{t-\tau}(\xvec{s}_v)^2 \, .
\end{eqnarray*}
It is worth noting that the spatial weighting parameters $w_{\tau,iv}$ might depend on the temporal lag. For $\tau = 0$, $w_{\tau,iv}$ describes the instantaneous spatial effect. Furthermore, one can rewrite the equation in matrix notation; that is,
\begin{eqnarray*}
\xvec{Y}_t & = & \text{diag}(\xvec{h}_t)^{1/2} \xvec{\varepsilon}_t \, , \\
\xvec{h}_t & = & \xvec{\alpha} + \sum_{\tau = 0}^{p} \xmat{W}_\tau  \, \text{diag}(\xvec{Y}_{t-\tau})\xvec{Y}_{t-\tau} \, .
\end{eqnarray*}
In the following paragraphs, we omit the index $t$.

The weighting matrix $\xmat{W}$ may depend on additional parameters. Possible choices for $\xmat{W}$ include, e.g.,
\begin{equation*}
\xmat{W} = \rho \tilde{\xmat{W}}\, , \;
\xmat{W} = \text{diag}(\rho_1,\ldots,\rho_n) \tilde{\xmat{W}}\, , \;
\xmat{W} = \text{diag}(\rho_1,\ldots,\rho_1,\ldots,\rho_r,\ldots,\rho_r) \tilde{\xmat{W}} \, ,
\end{equation*}
with a known weighting matrix $\tilde{\xmat{W}}$ or
\begin{equation*}
\xmat{W} = \rho \; \left(\lambda^{||\xvec{s}_i - \xvec{s}_j||} \right)_{i,j=1,\ldots,n} \, , \;
\xmat{W} = ( K(||\xvec{s}_i - \xvec{s}_j||; \xvec{\theta}) )_{i,j=1,\ldots,n}
\end{equation*}
with a decreasing function $K: [0,\infty) \rightarrow [0,\infty)$. Here, $||.||$ stands for the vector norm.
In Section \ref{sec:weighting_matrix}, we discuss some special weighting matrices in more detail.

Next, we focus on the conditions on the parameters such that the process is well defined. Initially, it is analyzed whether $\xvec{Y}$ is uniquely determined by $\xvec{\varepsilon}$.
Let
\begin{equation*}\label{eq:eta}
 \xvec{\eta} =  \left(
\begin{array}{c}
{\scriptscriptstyle \alpha_1 \varepsilon(\xvec{s}_1)^2 + \varepsilon(\xvec{s}_1)^2 \sum\limits_{v=1}^{n} w_{1v} \varepsilon(\xvec{s}_v)^2 \alpha_v} \\
{\scriptscriptstyle \alpha_2 \varepsilon(\xvec{s}_2)^2 + \varepsilon(\xvec{s}_2)^2 \sum\limits_{v=1}^{n} w_{2v} \varepsilon(\xvec{s}_v)^2 \alpha_v} \\
\vdots \\
{\scriptscriptstyle \alpha_n \varepsilon(\xvec{s}_n)^2 + \varepsilon(\xvec{s}_n)^2 \sum\limits_{v=1}^{n} w_{nv} \varepsilon(\xvec{s}_v)^2 \alpha_v} \\
\end{array}
\right)
\end{equation*}
and
\begin{equation*}\label{eq:A}
\xmat{A} = \text{diag}\left(\varepsilon(\xvec{s}_1)^2, \ldots, \varepsilon(\xvec{s}_n)^2\right) \xmat{W} \, , \quad \xvec{Y}^{(2)} = (Y(\xvec{s}_1)^2,\ldots,Y(\xvec{s}_n)^2)^\prime .
\end{equation*}

\begin{theorem}\label{th:f_eps}
Suppose that
\begin{equation}\label{eq:det}
\det\left( \xmat{I} - \xmat{A}^2 \right) \neq 0 \, .
\end{equation}
Then, there is one and only one $Y(\xvec{s}_1), \ldots, Y(\xvec{s}_n)$ that corresponds to each $\varepsilon(\xvec{s}_1), \ldots, \varepsilon(\xvec{s}_n)$. It holds that
\begin{equation}\label{eq:Y2}
\xvec{Y}^{(2)} = \left(\xmat{I} - \xmat{A}^2 \right)^{-1} \xvec{\eta} \; , \quad \xvec{h} = \xvec{\alpha} + \xmat{W} \left(\xmat{I} - \xmat{A}^2 \right)^{-1} \xvec{\eta}  \; , \text{and} \quad \xvec{Y} = \mathrm{diag}(\xvec{h})^{1/2}\xvec{\varepsilon} \, .
\end{equation}
\end{theorem}

Because of the complex dependence structure, i.e., $Y(\xvec{s}_i)$ depends on $Y(\xvec{s}_j)$ for all $i,j = 1, \ldots, n$ and vice versa, it turns out that the components of $\xvec{Y}^{(2)}$ are not necessarily nonnegative; thus, the square root of $h(\xvec{s}_i)$ might not exist. The choice of the weighting matrix $\xmat{W}$ affects whether all elements of the squared observations $\xvec{Y}^{(2)}$ are greater than or equal to zero. Moreover, this condition also depends on the realizations of the error vector $\xvec{\varepsilon}$. Therefore, we further analyze the required condition such that the components of $\xvec{Y}^{(2)}$ are nonnegative.

\begin{theorem}\label{th:spARCH}
Suppose that $\xvec{\alpha} \geq \xvec{0}$, $w_{ij} \geq 0$ for all $i,j=1,\ldots,n$, $w_{ii} = 0$ for all $i=1,\ldots,n$ and that $\det(\xmat{I} - \xmat{A}^2) \neq 0$. If all elements of the matrix $(\xmat{I} - \xmat{A}^2)^{-1}$ are nonnegative, then all components of $\xvec{Y}^{(2)}$ are nonnegative; i.e., $Y(\xvec{s}_i)^2 \ge 0$ for $i=1,\ldots,n$. Moreover, $h(\xvec{s}_i) \ge 0$ for $i=1,\ldots,n$.
\end{theorem}

In general, it seems to be difficult to check whether the condition given in Theorem \ref{th:spARCH} is fulfilled because it depends on both the weighting matrix $\xmat{W}$ and the error vector $\xvec{\varepsilon}$. However, in the important case in which $\xmat{W}$ is an upper or lower triangular matrix, the condition is always satisfied.

\begin{lemma}\label{lemma:triangular}
Suppose that $\xvec{\alpha} \geq \xvec{0}$, $w_{ij} \geq 0$ for all $i,j = 1,\ldots,n$ and $w_{ij} = 0$ for $1\le i \le j \le n$. All elements of the matrix $(\xmat{I} - \xmat{A}^2)^{-1}$ are then nonnegative.
\end{lemma}

Triangular matrices of spatial weights are of high practical relevance because the resulting spatial process can be observed as an oriented process. This means that the process evolves in a certain direction. In the case of a lower triangular matrix $\xmat{W}$, the location $\xvec{s}_1$ is regarded as the origin of the process. For the case of an arbitrary weighting matrix $\xmat{W}$, we need a criterion that can be more easily checked than that of Theorem \ref{th:spARCH}. Another possibility is given in the next lemma.

\begin{lemma}\label{lemma:nontriangular}
Suppose that $\xvec{\alpha} \geq \xvec{0}$, $w_{ij} \geq 0$ for all $i,j = 1,\ldots,n$ and $w_{ij} = 0$ for $i = j$. If
\begin{equation}\label{eq:assump}
\lim_{k \rightarrow \infty} \xmat{A}^{2k} = 0 \, ,
\end{equation}
then all elements of the matrix $(\xmat{I} - \xmat{A}^2)^{-1}$ are nonnegative.
\end{lemma}
It is worth noting that if $||\cdot||$ denotes some induced matrix norm, then \eqref{eq:assump} is fulfilled if $|| \xmat{A}^2 || < 1$ (cf. Theorem 18.2.19 of \cite{Harville97}).

To take a closer look at the condition in the above Lemma \ref{lemma:nontriangular}, we consider two simple examples.

\begin{example}\label{bsp:n2}
Initially, the simple spARCH process for $n=2$ is considered in more detail, which means that the process has exactly two observations $Y(\xvec{s}_1)$ and $Y(\xvec{s}_2)$ at the two locations $\xvec{s}_1$ and $\xvec{s}_2$. Simple calculations show that
\begin{equation*}
Y(\xvec{s}_i)^2 = \left\{
\begin{array}{ccc}
\varepsilon(\xvec{s}_1)^2 \frac{\alpha_1 + \alpha_2 w_{12} \varepsilon(\xvec{s}_2)^2}{1 - w_{12}w_{21} \varepsilon(\xvec{s}_1)^2 \varepsilon(\xvec{s}_2)^2} & \text{for} & i = 1 \\
\varepsilon(\xvec{s}_2)^2 \frac{\alpha_2 + \alpha_1 w_{21} \varepsilon(\xvec{s}_1)^2}{1 - w_{12}w_{21} \varepsilon(\xvec{s}_1)^2 \varepsilon(\xvec{s}_2)^2} & \text{for} & i = 2
\end{array} \right. \, .
\end{equation*}
These quantities are nonnegative if and only if $\xvec{\alpha} \geq \xvec{0}$, $w_{12} \ge 0$, $w_{21} \ge 0$ and
\begin{equation}\label{eq:assump2}
\varepsilon(\xvec{s}_1)^2 \varepsilon(\xvec{s}_2)^2 < \frac{1}{w_{12} w_{21}} \, .
\end{equation}
Consequently, $h(\xvec{s}_2) = \alpha_2 + w_{21} Y(\xvec{s}_1)^2 \geq 0$, and by analogy, $h(\xvec{s}_1) \geq 0$. Thus, all quantities are well defined.
Choosing in Lemma \ref{lemma:nontriangular} the norm $|| \xmat{B} ||_1 = \max\limits_{j} \sum_{i} |b_{ij}|$, it can be observed that the aforementioned condition is equivalent to condition \eqref{eq:assump2}.
\end{example}

This result shows that the support of $\varepsilon(\xvec{s}_i)^2$ must be bounded; otherwise, there arise problems with the interpretation of the model quantities. Hence, condition \eqref{eq:assump} must also be understood in this manner. This means that the support of the error quantities must be bounded in a certain manner. In the general case, the condition on the induced norm is more difficult to check. Therefore, we consider in the next example that the support of the error term is compact.

\begin{example}\label{bsp:finite_support}
Suppose that $\varepsilon(\xvec{s}_i)$ is taking values on a finite support. Let
$| \varepsilon(\xvec{s}_i) | \le a$ for all $i=1,\ldots,n$. Moreover, we will utilize the norm $|| \xmat{B} ||_1 = \max\limits_{j} \sum_{i} |b_{ij}|$. Now,
\begin{equation*}
\xmat{A}^2= \left( \varepsilon(\xvec{s}_i)^2 \sum_{v=1}^n w_{iv} w_{vj} \varepsilon(\xvec{s}_v)^2 \right)_{i,j=1,\ldots,n}
\end{equation*}
and
\begin{equation*}
|| \xmat{A}^2 ||_1 = \max_{1 \le j \le n} \sum_{i=1}^n \varepsilon(\xvec{s}_i)^2 \sum_{v=1}^n w_{iv} w_{vj} \varepsilon(\xvec{s}_v)^2 \le a^4
\max_{1 \le j \le n} \sum_{i=1}^n \sum_{v=1}^n w_{iv} w_{vj}  = a^4 ||\xmat{W}^2||_1 \, .
\end{equation*}
Thus, the norm is less than 1 if
\begin{equation*}
a \; < \; \frac{1}{\sqrt[4]{||\xmat{W}^2||_1}} \, .
\end{equation*}
It is worth noting that for $n=2$, we obtain the bound of Example \ref{bsp:n2}.
\end{example}

We observe a tradeoff between the weighting coefficients and the parameter $a$. To be precise, if the weighting coefficients increase, one would expect that the spatial autocorrelation of the squared observations would increase by the same magnitude. However, increasing values of the elements in $\xmat{W}$ imply smaller values of $a$, which reduces the extent of the spatial autocorrelation. We focus on this issue in more detail in Section \ref{sec:MC}. Below, the probability structure of $\xvec{Y}$ is derived.

Suppose that the assumptions of Theorem \ref{th:spARCH} are satisfied, with $\xvec{\alpha} > 0$, and that $\xvec{\varepsilon}$ is continuous with density function $f_{\xvec{\varepsilon}}$. Let $h_i = \alpha_i + \sum_{v=1, v \neq i}^n w_{iv} y_v^2$. Applying the transformation rule for random vectors (e.g., \citealt{Bickel15}), we obtain a density of
$\xvec{Y} = \text{diag}(\xvec{h})^{1/2} \xvec{\varepsilon} = f(\xvec{\varepsilon})$. Note that the transformation is one-to-one because if $\xvec{Y} = \text{diag}(\xvec{h})^{1/2} \xvec{\varepsilon} = f(\xvec{\varepsilon}) = \tilde{\xvec{Y}} = \text{diag}(\tilde{\xvec{h}})^{1/2} \tilde{\xvec{\varepsilon}} = f(\tilde{\xvec{\varepsilon}})$, it follows that $\xvec{h} = \xvec{h}(\xvec{Y}) = \xvec{h}(\tilde{\xvec{Y}}) = \tilde{\xvec{h}}$ and thus $\xvec{\varepsilon} = \tilde{\xvec{\varepsilon}}$ because $h(\xvec{s}_i) > 0$. We obtain that
\begin{eqnarray}
f_{\xvec{Y}}(\xvec{y}) & =  & f_{(Y(\xvec{s}_1), \ldots, Y(\xvec{s}_n))}(y_1, \ldots, y_n) \nonumber \\
& = & f_{(\varepsilon(\xvec{s}_1), \ldots, \varepsilon(\xvec{s}_n))}\left(\frac{y_1}{\sqrt{h}_1}, \ldots, \frac{y_n}{\sqrt{h}_n}\right) | \det\left( \left( \frac{\partial y_j/\sqrt{h_j}}{\partial y_i} \right)_{i,j=1,\ldots,n}\right) | \, . \label{eq:transformation}
\end{eqnarray}
Because
\begin{equation*}
\frac{\partial y_j/\sqrt{h_j}}{\partial y_i} = \left\{
\begin{array}{ccc}
1\, /\, \sqrt{h}_j & \mbox{for} & i=j \\
- \frac{y_i y_j}{h_j^{3/2}} w_{ji} & \mbox{for} & i \neq j
\end{array} \right. \, ,
\end{equation*}
it follows that
\begin{equation*}
| \det\left( \left( \frac{\partial y_j/\sqrt{h}_j}{\partial y_i} \right)_{i,j=1,\ldots,n}\right) | = \prod_{i=1}^n \frac{y_i^2}{h_i^{3/2}} \; \cdot \;  | \det\left( \text{diag}\left(\frac{h_1}{y_1^2},\ldots, \frac{h_n}{y_n^2}\right) + \xmat{W}^\prime \right) | \, .
\end{equation*}
The determinant of the sum of a diagonal matrix and an arbitrary matrix can be calculated as described in Theorem 13.7.3 of \cite{Harville97}.

In the special case of Example \ref{bsp:n2} ($n=2$), we obtain that $f_{(Y(\xvec{s}_1), Y(\xvec{s}_2))}(y_1,y_2)$
\begin{equation}\label{eq:sarch2}
= \frac{ \alpha_1 \alpha_2 + \alpha_1 w_{21} y_1^2 + \alpha_2 w_{12} y_2^2 }{(\alpha_1 + w_{12} y_2^2)^{3/2} (\alpha_2 + w_{21} y_1^2)^{3/2}} f_{(\varepsilon(\xvec{s}_1), \varepsilon(\xvec{s}_2))}\left( \frac{y_1}{\sqrt{\alpha_1+w_{12} y_2^2}}, \frac{y_2}{\sqrt{\alpha_2+w_{21} y_1^2}} \right) \, .
\end{equation}

Our next aim is to develop statements about the moments of $Y(\xvec{s}_i)$. To accomplish this, we shall assume that the error quantities are symmetric. There are various possibilities of defining symmetry for multivariate distributions (cf. \citealt{Serfling06}). Here, we consider sign-symmetric multivariate distributions.

\begin{theorem}\label{th:distr_sym}
Suppose that the assumptions of Theorem \ref{th:spARCH} are satisfied and that the distribution of $\xvec{\varepsilon}$ is sign-symmetric; i.e.,
\begin{equation*}\label{eq:gleich}
\xvec{\varepsilon} \stackrel{d}{=} \left( (-1)^{v_1} \varepsilon(\xvec{s}_1), \ldots, (-1)^{v_n} \varepsilon(\xvec{s}_n)\right)
\quad \text{for all} \quad v_1,\ldots,v_n \in \{0,1\} \, .
\end{equation*}
It then holds that the distribution of $\xvec{Y}$ is sign-symmetric as well.
\end{theorem}

It is important to note that $(\varepsilon(\xvec{s}_1),\ldots,\varepsilon(\xvec{s}_n))$ is sign-symmetric if the random variables $\varepsilon(\xvec{s}_1),\ldots,\varepsilon(\xvec{s}_n)$ are independent and if $\varepsilon(\xvec{s}_i)$ is symmetric about zero for all $i=1,\ldots,n$.

Next, we want to discuss the conditions under which the moments of $\xvec{Y}^{(2)}$ exist.
Using symmetry, it is proved that all odd moments are zero if the error variable is symmetric.
First, it is assumed that the weighting matrix $\xmat{W}$ is a triangular matrix.

\begin{lemma}\label{lemma:moments1}
Suppose that the assumptions of Theorem \ref{th:spARCH} are satisfied.  Let $n \ge 3$, $r \in \xset{N}$ and suppose that $E(\varepsilon(\xvec{s}_i)^{8r[(n-1)/2]}) < \infty$ for all $i=1,\ldots,n$. Let $w_{ij} \ge 0$ for $i,j=1,\ldots,n$ and $w_{ij} = 0$ for $1 \le i \le j \le n$; then, it holds that
\begin{itemize}
\item[a)]  $E( Y(\xvec{s}_i)^{2r} ) < \infty$ for all $i=1,\ldots,n$.
\item[b)]  If $\xvec{\varepsilon}$ is additionally sign-symmetric, then $E( Y(\xvec{s}_i)^{2v-1} ) = 0$ and $E( Y(\xvec{s}_i)^{2v-1} | Y(\xvec{s}_j), j=1,\ldots,n, j \neq i) = 0$ for $v=1,\ldots,r, i=1,\ldots,n$.
\end{itemize}
\end{lemma}

Below, we focus on the moments of the process in the case of an arbitrary weighting matrix.

\begin{theorem}\label{th:moments}
Suppose that the assumptions of Theorem \ref{th:distr_sym} are satisfied. Let $||.||$ denote some induced matrix norm. Let $r \in \xset{N}$, and suppose that $E(\varepsilon(\xvec{s}_i)^{2r}) < \infty$ for all $i=1,\ldots,n$.
\begin{itemize}
\item[a)] If there exists a constant $\lambda > 0$ such that
\begin{equation*}\label{Kond}
||(\xmat{I} - \xmat{A}^2)^{-1}|| \le \lambda \, ,
\end{equation*}
then it holds that
\begin{itemize}
\item[$a_1$)]  $E( Y(\xvec{s}_i)^{2r} ) < \infty$ for all $i=1,\ldots,n$.
\item[$a_2$)]  $E( Y(\xvec{s}_i)^{2v-1} ) = 0$ and $E( Y(\xvec{s}_i)^{2v-1} | Y(\xvec{s}_j), j=1,\ldots,n, j \neq i) = 0$ for $v=1,\ldots,r, i=1,\ldots,n$.
\end{itemize}

\item[b)] If there exists $0 < \lambda < 1$ such that
\begin{equation*}
 ||\xmat{A}^2|| \le \lambda < 1 \, ,
\end{equation*}
then $||(\xmat{I} - \xmat{A}^2)^{-1}||$ is bounded.
\end{itemize}
\end{theorem}

The moments of $\varepsilon(\xvec{s}_i)$ are of course bounded if we assume that the support of $\varepsilon(\xvec{s}_i)$ is bounded.

It is worth noting that one important property of the classical, temporal GARCH approach is not fulfilled for each specification of $\xmat{W}$. Generally, it does not hold that $h(\xvec{s}_1)$ is equal to $E(Y(\xvec{s}_1)^2 | Y(\xvec{s}_2),\ldots, Y(\xvec{s}_n) )$. To prove this, we consider the simple case of Example \ref{bsp:n2} ($n=2$). In that case,
\begin{equation*}
E(Y(\xvec{s}_1)^2 | Y(\xvec{s}_2) ) = (\alpha_1 + w_{12} Y(\xvec{s}_2)^2) \;  E(\varepsilon(\xvec{s}_1)^2 | Y(\xvec{s}_2) ) \, .
\end{equation*}
The problem lies in the fact that $\varepsilon(\xvec{s}_1)$ and $Y(\xvec{s}_2)$ are not independent; thus, $E(\varepsilon(\xvec{s}_1)^2 | Y(\xvec{s}_2) )$ does not have to be equal to $E(\varepsilon(\xvec{s}_1)^2)$.
If the conditions of Theorem \ref{th:moments} are fulfilled and if $\varepsilon(\xvec{s}_1)$ and $\varepsilon(\xvec{s}_2)$ are independent, it follows with \eqref{eq:sarch2} that
\begin{eqnarray*}
&&  E( Y(\xvec{s}_1)^2 | Y(\xvec{s}_2) = y_2 )   =  \\
&&
\footnotesize{
\quad \frac{1}{f_{Y(\xvec{s}_2)}(y_2)} \int\limits_{-\infty}^\infty y_1^2 \frac{ \alpha_1 \alpha_2 + \alpha_1 w_{21} y_1^2 + \alpha_2 w_{12} y_2^2 }{(\alpha_1 + w_{12} y_2^2)^{3/2} (\alpha_2 + w_{21} y_1^2)^{3/2}} \;f_{\varepsilon(\xvec{s}_1)}\left( \frac{y_1}{\sqrt{\alpha_1+w_{12} y_2^2}} \right) \;
f_{\varepsilon(\xvec{s}_2)}\left( \frac{y_2}{\sqrt{\alpha_2+w_{21} y_1^2}} \right) \; dy_1
} \; .
\end{eqnarray*}

\begin{figure}
  \centering
    \includegraphics[width=0.65\textwidth,natwidth=500,natheight=500]{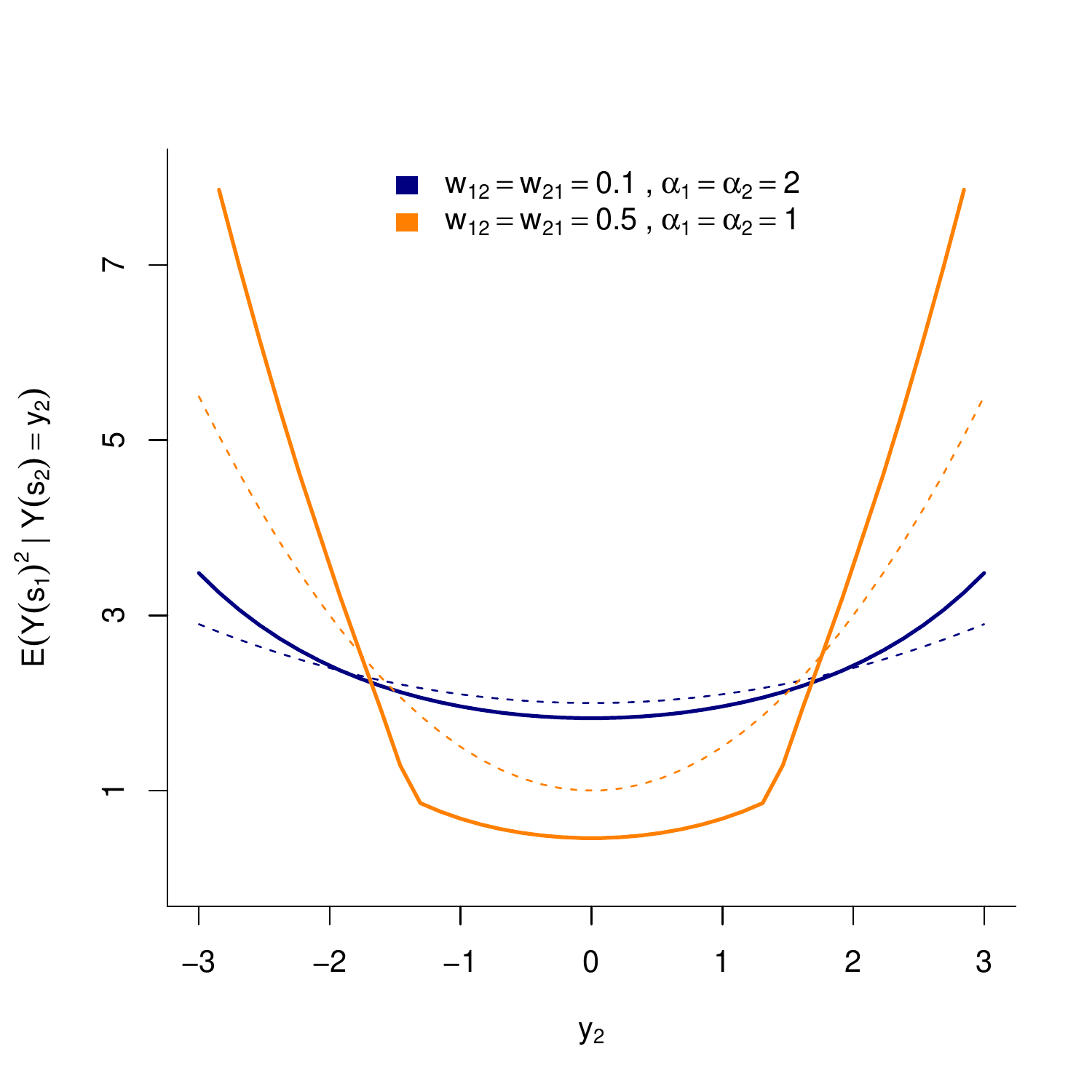}
  \caption{Conditional expectation of $Y(\xvec{s}_1)^2$ given $y_2$ for $n=2$, where $E(Y(\xvec{s}_1)^2 | Y(\xvec{s}_2) = y_2^2)$ is plotted as a solid line and $h(\xvec{s}_1)$ as a dashed line.}\label{fig:cond_exp}
\end{figure}

In Figure \ref{fig:cond_exp}, the conditional expectation of $Y(\xvec{s}_1)^2$ given $Y(\xvec{s}_2)$ is plotted together with $h(\xvec{s}_1)$ for two different specifications of $\xmat{W}$ and $\xvec{\alpha}$. Obviously, $E( Y(\xvec{s}_1)^2 | Y(\xvec{s}_2) = y_2 )$ differs such that the greater the difference from $h(\xvec{s}_1)$,  the larger the chosen elements of $\xmat{W}$. Certainly, the difference between the conditional expectation and $h(\xvec{s}_1)$ vanishes for $\xmat{W}=\xmat{0}$. However, we find that this classical property of an ARCH process, namely, that $E(\varepsilon(\xvec{s}_1)^2 | Y(\xvec{s}_2) ) = h(\xvec{s}_1)$, is fulfilled in the case of a triangular weighting matrix. To summarize, the conditional variance given the neighboring observations depends on these neighboring observations. It is important to note that this is not the case for linear spatial models (cf. \citealt{Cressie93}). Hence, the new model is much more flexible.

\begin{theorem}\label{th:triangularW}
Suppose that the assumptions of Theorem \ref{th:spARCH} are satisfied. Let $w_{ij} \ge 0$ for
$i,j = 1,\ldots,n$ and $w_{ij} = 0$ for $1 \le i \le j \le n$. Suppose that $E(\varepsilon(\xvec{s}_i)^{8[(n-1)/2]}) < \infty$ for all $i=1,\ldots,n$ and let $\varepsilon(\xvec{s}_1),\ldots, \varepsilon(\xvec{s}_n)$ be independent. It then holds for each $k\in \{1,\ldots,n\}$ that
\begin{equation*}
E(Y(\xvec{s}_k)^2 | Y(\xvec{s}_j), j = 1, \ldots, k-1) = h_k \, .
\end{equation*}
\end{theorem}

Principally, it is not necessary that the matrix $\xmat{W}$ of spatial weights be a triangular matrix, but there should exist a permutation matrix $\xmat{P}$ such that $\ddot{\xmat{W}} = \xmat{P}\xmat{W}\xmat{P}'$ is triangular. In this case, the observations also must be permuted; i.e., the permuted vector of observations is $\ddot{\xvec{Y}} = \xmat{P}\xvec{Y}$.

Furthermore, one may see that
\begin{equation*}
E(Y(\xvec{s}_k)^2 | Y(\xvec{s}_j), j = k+1, \ldots, n) = h_k \, ,
\end{equation*}
if the weighting matrix $\xmat{W}$ is a strictly upper triangular matrix. In the following section, we take a closer look at two different specifications of the weighting matrix $\xmat{W}$.

\subsection{Choice of the Weighting Matrix $\xmat{W}$}\label{sec:weighting_matrix}

In this section, we suggest two different specifications of the matrix $\xmat{W}$ of spatial weights to adapt the process to various situations. In particular, the second matrix is a triangular matrix; i.e., for this specification, the support of $\xvec{\varepsilon}$ must not be bounded, and Theorem \ref{th:triangularW} can be applied.

First, we present a possible method to model more than one lag in space. Assume that the set $\zeta(\delta, \xvec{s}_i) = \{j : ||\xvec{s}_i - \xvec{s}_j|| \in (\delta - c, \delta]\}$ consists of all locations $j$ for which the distance from location $\xvec{s}_i$ is between $\delta - c$ and $\delta$. The distance is measured by some predefined metric $||\xvec{a}-\xvec{b}||$ on the considered space induced by an arbitrary norm $||\cdot||$. The spatial lag constant $c$ is equivalent to the time period of one lag in the temporal setting, which could be one day, week, year, et cetera. In the spatial setting, the constant $c$ must be chosen according to specific requirements of the process, e.g., 1 $\mu m$ - 1 $mm$ (microbiology), 1 $cm$ - 1 $m$ (materials science) or 1 $km$ - 100 $km$ (macroeconomics). Finally, the weighting matrix $\xmat{W}$ is based on an arbitrarily chosen matrix $\tilde{\xmat{W}}$ fulfilling the assumptions introduced above, such as the binary contiguity matrix, nearest-neighbor matrix, or inverse-distance matrix (cf. \citealt{Elhorst10}). The elements of $\xmat{W}$ can be specified as
\begin{equation} \label{eq:W_spARCH_p}
w_{ij} = \left\{
\begin{array}{ccc}
\tilde{w}_{ij} \, \sum\limits_{k=1}^{p}  \rho_k \xset{1}_{\zeta(kc, \xvec{s}_i)}(j) & \text{for} & i \neq j \\
0 & \text{for} & i = j
\end{array} \right. \qquad \forall i,j=1,\ldots,n \, ,
\end{equation}
where $\xset{1}_{A}$ is the indicator function on the set $A$. Hence, two locations $\xvec{s}_i$ and $\xvec{s}_j$ are weighted by $\rho_1 \tilde{w}_{ij}$ if they are first lag neighbors; i.e., the distance between $\xvec{s}_i$ and $\xvec{s}_j$ lies between zero and $c$. Moreover, these two locations are weighted by $\rho_2 \tilde{w}_{ij}$ if the distance is between $c$ and $2c$. In this manner, as many as $p \in \{1,2, \ldots, \left\lceil c^{-1} \max_{ij}||\xvec{s}_i-\xvec{s}_j||\right\rceil\}$ spatial lags can be included in the process.  We refer to this specification of $\xmat{W}$ as the spatial ARCH process of order $p$ (spARCH($p$)). Because the matrix $\tilde{\xmat{W}}$ is assumed to be known, it remains to estimate only $p$ spatial autoregressive parameters $\rho_1,\ldots,\rho_p$.

Second, an example to model processes with some direction is presented. For instance, oriented processes could spread from some center/origin into every direction of the considered space (e.g., epidemiology or disease mapping), or the process could evolve in one direction, e.g., from north to south (e.g., ocean currents or wind speed). In particular, we focus on the first case of an oriented process. Therefore, assume that there is some known origin $\xvec{s}_0$ of the spatial process. It is worth noting that the origin could also be estimated. Without loss of generality, one can order the locations $\xvec{s}_1,\ldots,\xvec{s}_n$  with respect to the distance from the center. Thus,
\begin{equation*}
0 < ||\xvec{s}_1 - \xvec{s}_0|| \leq ||\xvec{s}_2 - \xvec{s}_0|| \leq \ldots \leq ||\xvec{s}_n - \xvec{s}_0|| \, .
\end{equation*}
Assuming additionally that each location is influenced only by the locations closer to the center leads to an upper triangular representation of $\xmat{W}$; i.e.,
\begin{equation} \label{eq:W_oriented}
w_{ij} = \left\{
\begin{array}{ccc}
\tilde{w}_{ij}  & \text{for} & ||\xvec{s}_i - \xvec{s}_0|| < ||\xvec{s}_j - \xvec{s}_0||  \\
0 & \multicolumn{2}{c}{\text{otherwise}}
\end{array} \right.
= \left\{
\begin{array}{ccc}
\tilde{w}_{ij}  & \text{for} & i < j  \\
0 & \text{for}  & i \geq j
\end{array} \right. \, .
\end{equation}

\begin{figure}
  \centering
    \includegraphics[width=0.45\textwidth,natwidth=500,natheight=500]{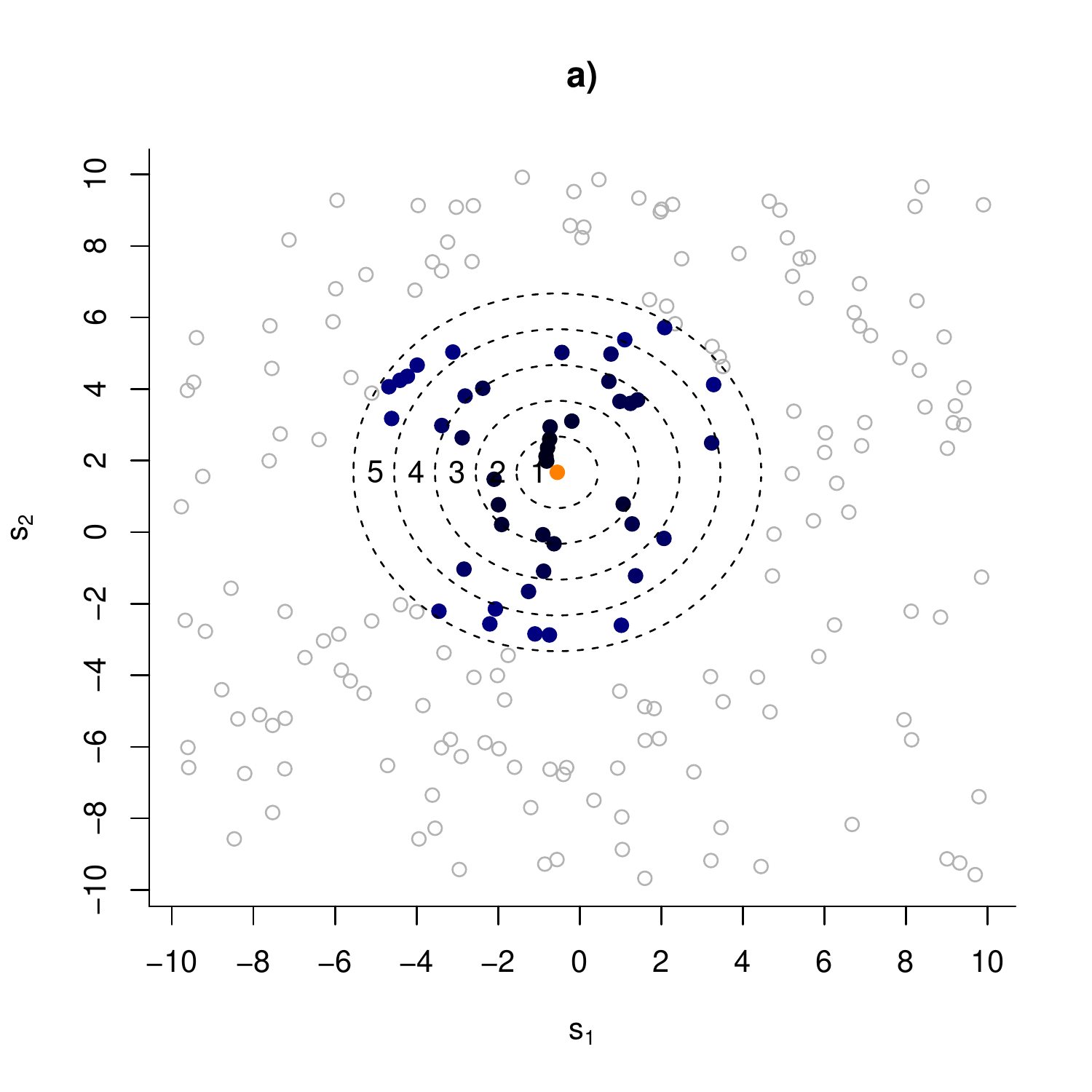}
    \includegraphics[width=0.45\textwidth,natwidth=500,natheight=500]{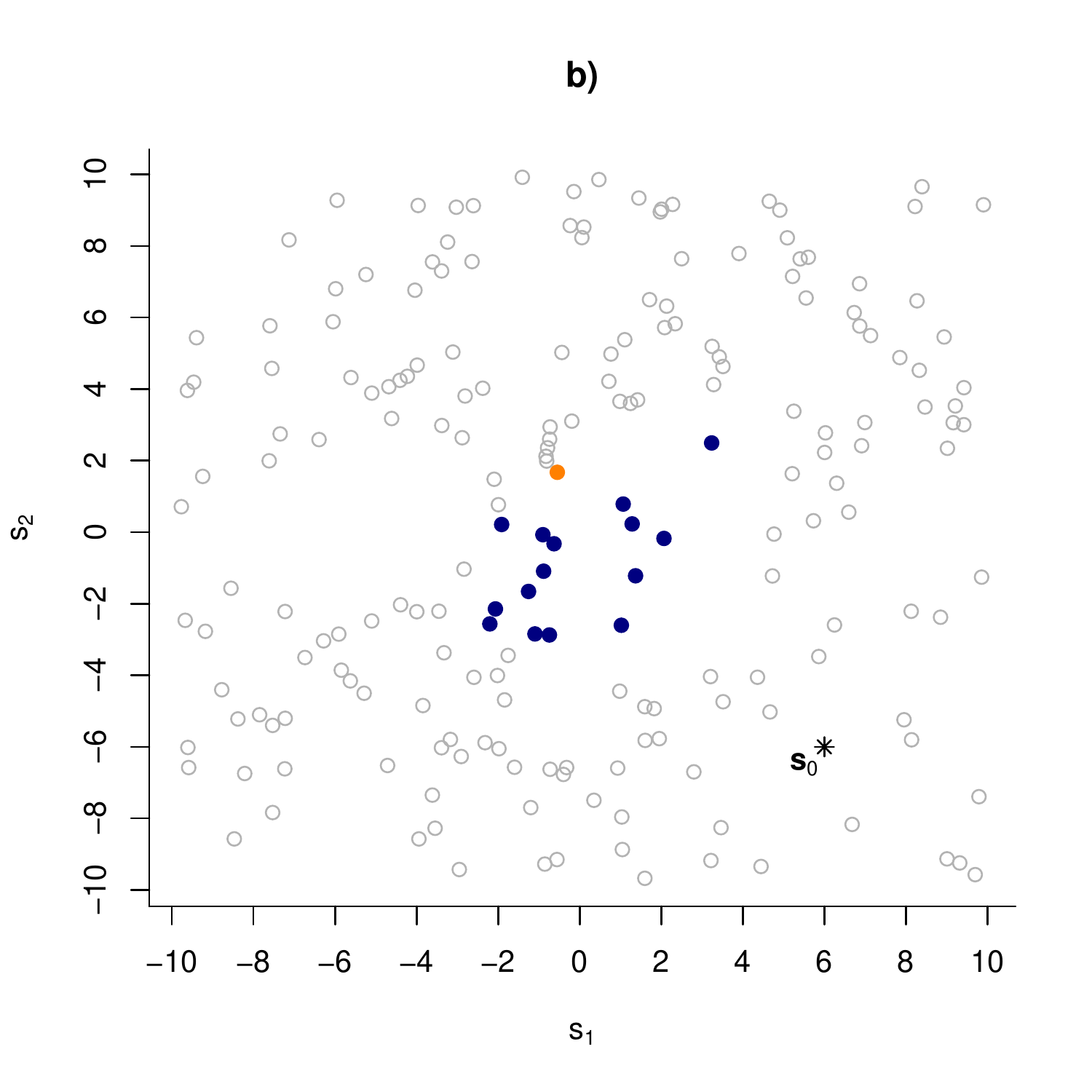}
  \caption{Representation of the positive elements in $\xmat{W}$ (colored in grey, filled dots) for some location $i$ (colored in red) regarding a) a spARCH(5) process and b) an oriented process with center $\xvec{s}_0$. All locations that influence the red location $i$ are colored in grey, whereas all other location having no influence on $i$ are drawn as empty circles. Matrix $\tilde{\xmat{W}}$ is chosen as the binary matrix of the 50 nearest neighbors. Moreover, the 200 locations $\xvec{s} = (s_1, s_2)$ result from a continuous process in two-dimensional space.}\label{fig:choice_W}
\end{figure}

Both examples of $\xmat{W}$ are illustrated in Figure \ref{fig:choice_W}. First, the proposed spARCH($p$) is illustrated in the left-hand figure a) for $p=5$. The positive weights of the $i$-th row of $\xmat{W}$ are drawn with filled circles; i.e., the variance of the observations at all locations, which are drawn with filled circles, influence the variance of the observation at location $i$ colored in red. The spatial lag constant $c$ is assumed to be $1$, and the distance between the locations is measured using the Euclidean norm.
Second, we illustrate the proposed oriented process on the right-hand side of Figure \ref{fig:choice_W}. The point of origin $\xvec{s}_0$ is drawn as a star, such that one may see that only locations closer to $\xvec{s}_0$ have an influence on the location $i$. Regarding both cases a) and b), we choose $\tilde{\xmat{W}}$ as the $q$-nearest-neighbor matrix, where $q=50$.

Finally, we provide the link to classical heteroscedastic time-series models and other propositions of spatial ARCH models in Table \ref{table:summary}. In particular, we show how the parameters and the spatial weighting matrix must be chosen to transfer the introduced model to the classical ARCH($p$) process proposed by \cite{Engle82}. It is worth noting that the support of the error distribution does not have to be bounded because $\xmat{W}$ is triangular.

\begin{table}
    \begin{center}
    \caption{Summary of several covered settings.}\label{table:summary}
    \begin{footnotesize}
    \begin{tabular}{lcclc}
    \hline \hline
    Model & $q$ & $D_s$ & $\xmat{W}$ & triangular \\
    \hline
    \multicolumn{5}{l}{\emph{time-series models}} \\
    $\quad$ARCH(1) \cite{Engle82} & 1 & $\xset{Z}$ & $\left(\alpha\xset{1}_{\{s_i - s_j = 1\}}\right)_{i,j=1,\ldots,n}$ & $\checkmark$ \\
    $\quad$ARCH($p$) \cite{Engle82} & 1 & $\xset{Z}$ & $\left(\sum_{k=1}^{p}\alpha_k\xset{1}_{\{s_i - s_j = k\}}\right)_{i,j=1,\ldots,n}$ & $\checkmark$  \\
    \multicolumn{5}{l}{\emph{spatiotemporal models}} \\
    $\quad$spatial ARCH & & & & \\
    $\quad$\cite{Borovkova12} & 1 & $\xset{Z}$ & $\left((a_{1,i} + a_{2,i} w_{ij}) \xset{1}_{\{s_i - s_j = 1\}}\right)_{i,j=1,\ldots,n}$ & $\checkmark$ \\
    \multicolumn{5}{l}{\emph{spatial models}} \\
    $\quad$SARCH(1) \cite{Bera04} & $2,3$ & $\xset{Z}^q$, $\xset{R}^q$ & $\left(\alpha_1 w_{ij}^2\right)_{i,j=1,\ldots,n}$ & \\
    \multicolumn{5}{l}{\emph{new propositions (multidimensional)}} \\
    $\quad$spARCH($p$) & $\geq 1$ & $\xset{Z}^q$, $\xset{R}^q$ & cf. eq. \eqref{eq:W_spARCH_p} & \\
    $\quad$oriented & $\geq 1$ & $\xset{Z}^q$, $\xset{R}^q$ & cf. eq. \eqref{eq:W_oriented} & $\checkmark$ \\
    \hline
    \end{tabular}
    \end{footnotesize}
    \end{center}
\end{table}

\section{Statistical Inference}\label{sec:inference}

To date, the weighting matrix $\xmat{W}$ has mostly been chosen to be an arbitrary matrix with nonnegative elements and zeros on the main diagonal. To ensure that $h(\xvec{s}_i)$ and $Y(\xvec{s}_i)$ are nonnegative, the weights must fulfill an additional condition as shown, e.g., in Theorem \ref{th:spARCH} and Example \ref{bsp:finite_support}. These conditions connect the weights with the support of $\xvec{\varepsilon}(\xvec{s}_i)$. In applications, the weighting matrix $\xmat{W}$ may depend on additional parameters as discussed earlier.


First, we consider the model
\begin{equation*}
h(\xvec{s}_i) = \alpha + \rho \sum_{v=1}^{i-1} \tilde{w}_{iv} Y(\xvec{s}_v)^2 , \quad i=1,\ldots,n
\end{equation*}
with $\tilde{w}_{iv} \ge 0$ for $i, v = 1,\ldots, n$ and $\tilde{w}_{iv} = 0$ for $1 \le i \le v \le n$. Thus, $\xmat{W}$ is chosen as a lower triangular matrix. It is assumed that $\alpha > 0$ and $\rho > 0$.

Suppose that $\varepsilon(\xvec{s}_1),\ldots, \varepsilon(\xvec{s}_n)$ are independent and identically distributed. Let $f_\varepsilon$ denote its density function and let $f_\varepsilon$ be differentiable.  Moreover, let $\xvec{y} = (y_1,\ldots,y_n)^\prime$ be the vector of observations and $h_i = h(\xvec{s}_i; \xvec{y})$. Using \eqref{eq:transformation}, the density of $\xvec{Y}$ is given by
\begin{eqnarray*}
f_{\xvec{Y}}(\xvec{y}) & = & \prod_{i=1}^n \left( f_{\varepsilon}\left( \frac{y_i}{\sqrt{h}_i} \right) \; \frac{1}{\sqrt{h_i}} \right) = \prod_{i=1}^n f_{Y(\xvec{s}_i)|Y(\xvec{s}_{i-1}),\ldots,Y(\xvec{s}_1)}(y_i|y_{i-1},\ldots,y_1)
\end{eqnarray*}
and
\begin{eqnarray*}
\log(f_{\xvec{Y}}(\xvec{y})) & = & \sum_{i=1}^n \left( \log\left(f_{\varepsilon}\left( \frac{y_i}{\sqrt{h}_i} \right)\right) - \frac{1}{2} \log(h_i) \right) \, .
\end{eqnarray*}

Let $\tilde{f} = f_{\varepsilon}^\prime/f_\varepsilon$. Putting the partial derivatives of $\log(f_{\xvec{Y}}(\xvec{y}; \alpha, \rho))$ with respect to $\alpha$ and $\rho$ equal to zero, we obtain the estimators $\hat{\alpha}$ and $\hat{\rho}$ that satisfy
\begin{eqnarray}\label{para1}
\sum_{i=1}^n \frac{1}{\hat{\alpha} + \hat{\rho} A_i} & = & - \sum_{i=1}^n \frac{y_i}{(\hat{\alpha} + \hat{\rho} A_i)^{3/2}} \; \tilde{f}\left(\frac{y_i}{\sqrt{\hat{\alpha} + \hat{\rho} A_i}}\right) ,  \\
\sum_{i=2}^n \frac{A_i}{\hat{\alpha} + \hat{\rho} A_i} & = & - \sum_{i=1}^n \frac{A_i y_i}{(\hat{\alpha} + \hat{\rho} A_i)^{3/2}} \; \tilde{f}\left(\frac{y_i}{\sqrt{\hat{\alpha} + \hat{\rho} A_i}}\right)
\label{para2}
\end{eqnarray}
with $A_i = \sum_{v=1}^{i-1} \tilde{w}_{iv} y_v^2$ for $i = 1, \ldots, n$.
If the corresponding information matrix $\xmat{B}_n$ is positive definite, then the results of \cite{Crowder76} can be applied. It follows that there is a unique solution of \eqref{para1} and \eqref{para2}. The estimators $\hat{\alpha}$ and $\hat{\rho}$ are consistent, and $(\hat{\alpha}, \hat{\rho})$ is approximately distributed as ${\cal N}_2(\xvec{0}, \xmat{B}_n^{-1})$. This result can be used for testing the hypotheses on the parameters $\alpha$ and $\rho$.

For instance, assuming $f_\varepsilon$ to be the standard normal distribution. Then, it follows that $\tilde{f}(x)=-x$, and the information matrix is given by
\begin{equation*}
\xmat{B}_n = - E\left( \begin{array}{cc}
- \frac{1}{2} \sum\limits_{i=1}^n \frac{1}{(\alpha+\rho A_i)^2} + \sum\limits_{i=1}^n \frac{y_i^2}{(\alpha+\rho A_i)^3} & \quad \quad - \frac{1}{2} \sum\limits_{i=1}^n \frac{A_i}{(\alpha+\rho A_i)^2} + \sum\limits_{i=1}^n \frac{A_i y_i^2}{(\alpha+\rho A_i)^3}\\[0.5cm]
- \frac{1}{2} \sum\limits_{i=1}^n \frac{A_i}{(\alpha+\rho A_i)^2} + \sum\limits_{i=1}^n \frac{A_i y_i^2}{(\alpha+\rho A_i)^3} & \quad \quad - \frac{1}{2} \sum\limits_{i=1}^n \frac{A_i^2}{(\alpha+\rho A_i)^2} + \sum\limits_{i=1}^n \frac{A_i^2 y_i^2}{(\alpha+\rho A_i)^3}
\end{array} \right) \, .
\end{equation*}

These results can be easily extended to more general models, such as the approach described in (\ref{eq:W_spARCH_p}). Moreover, in this section, we focused on lower triangular matrices, but all of the results presented above also hold for upper triangular matrices.


Next, we want to consider a model in which the weight matrix is neither a lower nor an upper triangular matrix. Let
\begin{equation*}
h(\xvec{s}_i) = \alpha + \rho \sum_{v=1}^{n} \tilde{w}_{iv} Y(\xvec{s}_v)^2 , \quad i=1,\ldots,n
\end{equation*}
with $\tilde{w}_{iv} \ge 0$ for $i, v = 1,\ldots, n$ and $\tilde{w}_{ii} = 0$ for $1 \le i \le n$. It is assumed that $\alpha > 0$ and $\rho > 0$. For these settings, the determinant
\begin{equation*}
| \det\left( \left( \frac{\partial y_j/\sqrt{h}_j}{\partial y_i} \right)_{i,j=1,\ldots,n}\right) | \,
\end{equation*}
must be computed. For practical applications, it is much easier to compute the logarithm of this determinant; i.e.,
\begin{equation*}
\log | \det\left( \left( \frac{\partial y_j/\sqrt{h}_j}{\partial y_i} \right)_{i,j=1,\ldots,n}\right) | = \sum_{i=1}^n \left(2\log{y_i} - \frac{3}{2}\log{h_i}\right) \; + \;  \sum_{i=1}^{n} \log |\lambda_i| \, ,
\end{equation*}
where $\lambda_i$ is the $i$-th eigenvalue of $\left( \text{diag}\left(\frac{h_1}{y_1^2},\ldots, \frac{h_n}{y_n^2}\right) + \rho \xmat{W}^\prime \right)$.
In addition, it is important to note that the weighting matrix is usually sparse, and there are positive weights up to the $k$-th subdiagonal, where $k = \max\{|i-j| : w_{ij} > 0\}$. If the locations are well ordered (e.g., by the distance to an arbitrarily chosen location), $k$ is much smaller than $n$.

\section{Applications}\label{sec:empirical}

In the following section, the focus is on applications of the suggested spatial ARCH model. In particular, we extend the well-known spatial autoregressive process by assuming conditional heteroscedastic residuals. Finally, the model parameters of such a model are estimated for a real data example. In the ensuing Section \ref{sec:MC}, we analyze the performance of the estimators in more detail by reporting the results of an extensive simulation study.

\subsection{Spatial Autoregressive Process with Conditional Hetero\-sce\-dastic Residuals: SARspARCH}

For the definition of the spatial autoregressive process, we must introduce a further matrix $\xmat{B}$ of spatial weights. This matrix $\xmat{B}$ could differ from the aforementioned weighting matrix $\xmat{W}$. However, it is also assumed that $\xmat{B}$ is non-stochastic and nonnegative with zeros on the main diagonal. Furthermore, let $\lambda$ denote the spatial autoregressive coefficient and $\mu$ be the mean parameter. The model is then defined as follows:
\begin{equation}
    \xvec{Y} = \mu\xvec{1} + \lambda \xmat{B} \xvec{Y} + \xvec{\xi} \; , \text{i.e.} \quad \xvec{Y} =  (\xmat{I} - \lambda \xmat{B})^{-1}(\mu \xvec{1} + \xvec{\xi})\, . \label{eq:SARspARCH1}
\end{equation}
The vector of disturbances $\xvec{\xi} = \left(\xi_1, \ldots, \xi_n \right)$ follows a spatial ARCH according to the suggested model in \eqref{eq:initial}. Consequently, the error process is given by
\begin{eqnarray*}
\xvec{\xi} = \text{diag}(\xvec{h})^{1/2} \xvec{\varepsilon} \,
\end{eqnarray*}
and
\begin{eqnarray}
\xvec{h} = \xvec{\alpha} + \xmat{W}  \, \text{diag}(\xvec{\xi})\xvec{\xi} \, . \label{eq:SARspARCH2}
\end{eqnarray}

In Figure \ref{fig:simulated_SARspARCH}, we plotted four different simulated spatial models to illustrate the behavior of these processes and compare them with respect to their properties. Moreover, the respective spatial autocorrelation functions (ACF) are shown in Figure \ref{fig:ACF_simulated_SARspARCH}. For the simulation, the spatial domain is assumed to be a lattice; i.e., $D_{\xvec{s}} = \{(i,j) \in \xset{Z}^2: i,j = 1, \ldots, d \}$. In plot (a), the innovations $\xvec{\varepsilon}$ truncated on the interval $[-a,a]$ are shown. The respective bound $a$ results from the choice of the weighting matrix $\xmat{W}$. In particular, this matrix is assumed to be the product of the parameter $\rho$ and a known weighting matrix $\tilde{\xmat{W}}$; i.e., $\xmat{W} = \rho \tilde{\xmat{W}}$. This setting was discussed in Section \ref{sec:weighting_matrix} as a spARCH($1$) model, where $c$ equals $1$, and the considered metric is induced by the maximum norm. Moreover, the known matrix $\tilde{\xmat{W}}$ of spatial weights is a classical row-standardized Rooks contiguity matrix, and $\xvec{\alpha} = \alpha_0 \xvec{1}$. In the next plot (b) of Figure \ref{fig:simulated_SARspARCH}, the spatial autoregressive process with white noise $\xvec{\varepsilon}$ is plotted. The simulation shows the classical behavior of a spatial autoregressive process; i.e., one can observe clusters of high and low values. The proposed spARCH model and SARspARCH model are presented in the second row of Figure \ref{fig:simulated_SARspARCH}. On the left-hand side, the simulation of the spatial ARCH process $\xvec{\xi}$ is shown. Obviously, that process differs from the white noise process in (a). The clusters of high and low variance are characterized by the luminance of the colors. Thus, the variance is low in areas where the observations have a light color, and the variance is high in areas of deep colors. This is supported by the ACF function in Figure \ref{fig:ACF_simulated_SARspARCH}, where the squared observations are positively correlated. The simulation of the SARspARCH process according to \eqref{eq:SARspARCH1} and \eqref{eq:SARspARCH2} yields the last image (d) in Figure \ref{fig:simulated_SARspARCH}.

\begin{figure}
  \begin{center}
  \includegraphics[width=0.4\textwidth,natwidth=500,natheight=500]{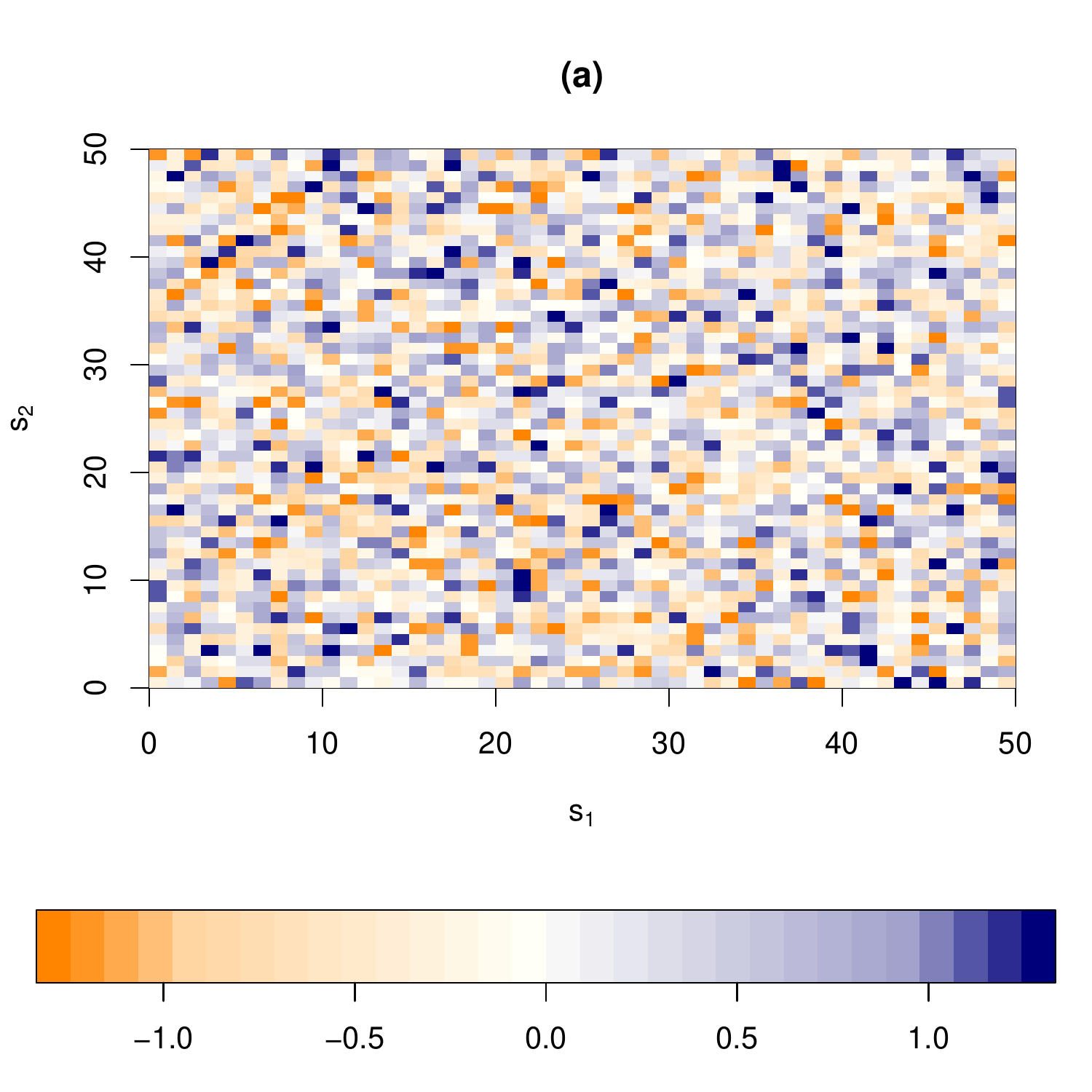}
  \includegraphics[width=0.4\textwidth,natwidth=500,natheight=500]{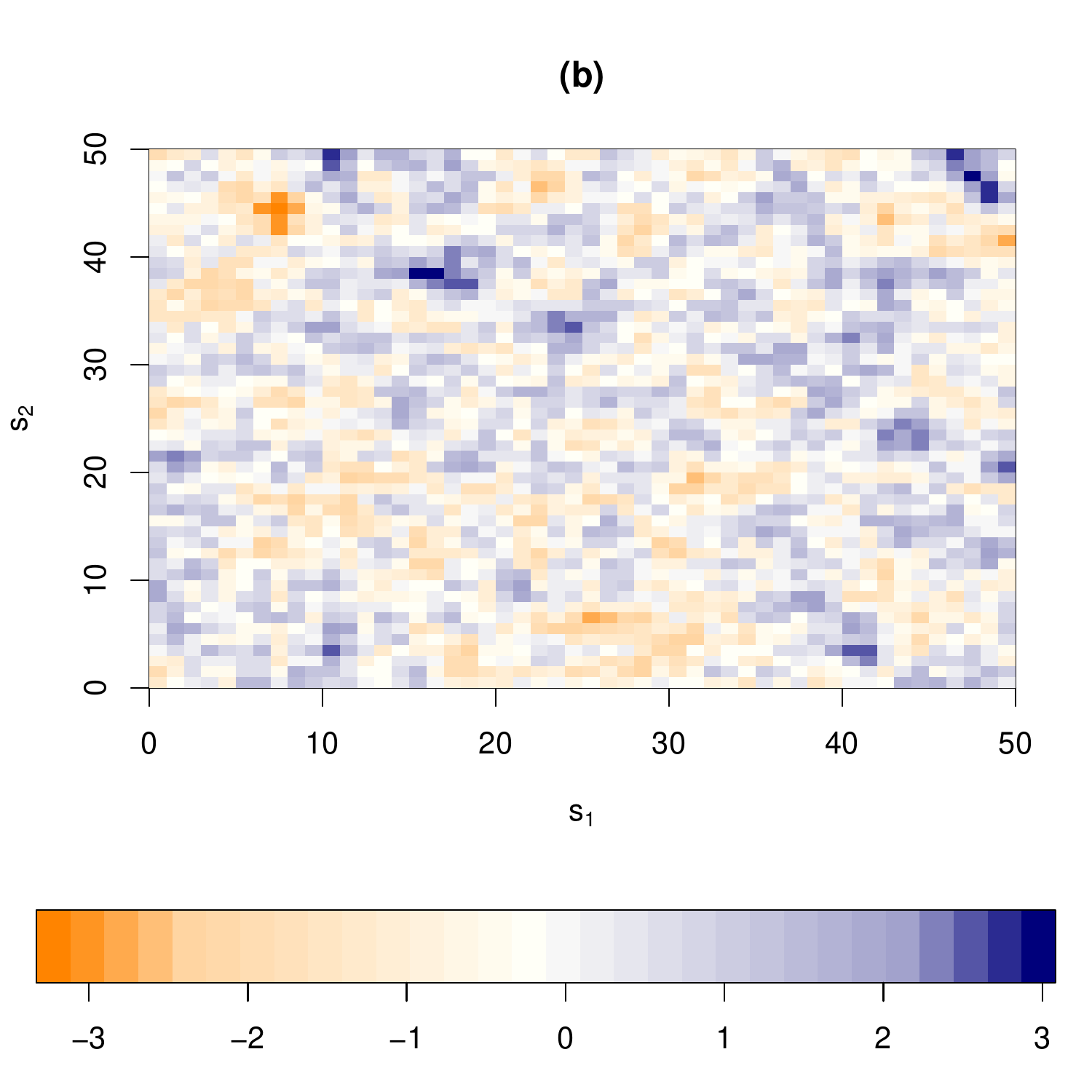}\\
  \includegraphics[width=0.4\textwidth,natwidth=500,natheight=500]{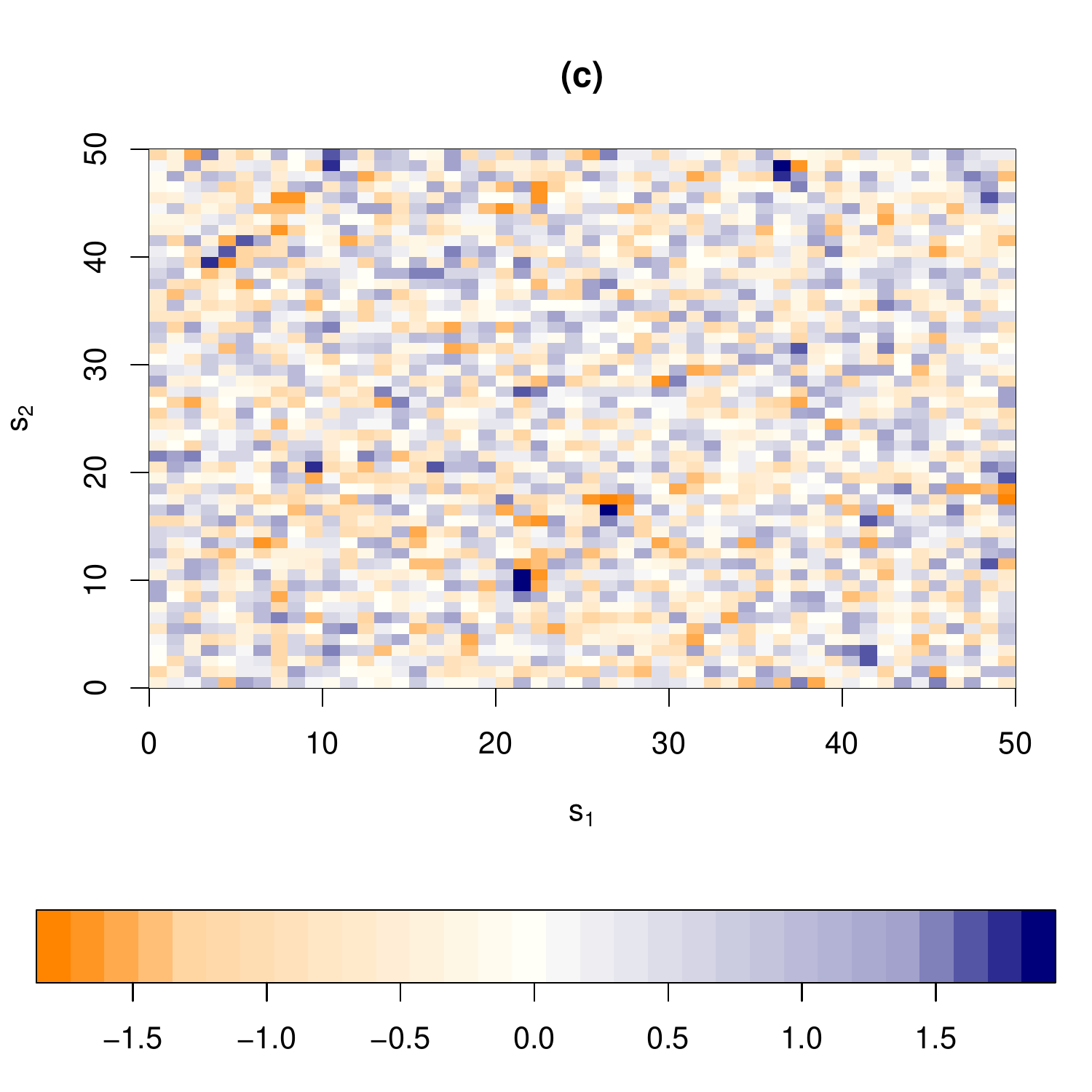}
  \includegraphics[width=0.4\textwidth,natwidth=500,natheight=500]{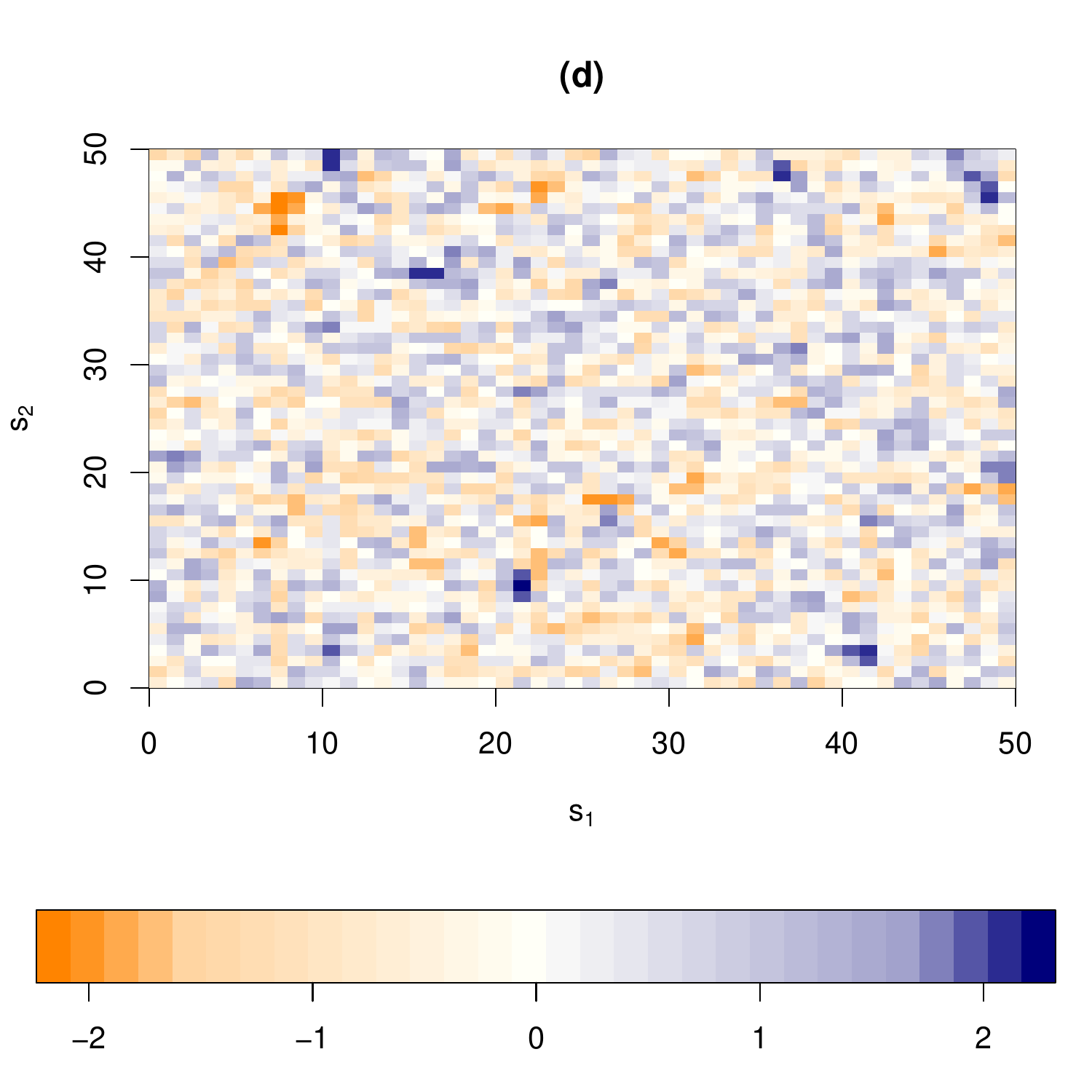}
  \end{center}
  \caption{Simulated spatial white noise process (a) truncated on $[-a,a]$, spatial autoregressive process (b), spatial ARCH process (c), spatial autoregressive process with spatial ARCH errors (d), where $d=50$, $\lambda = 0.8$, $\rho=0.5$ ($\rightsquigarrow a = 1.334$), $\mu = 0$, $\alpha_0 = 0.1$ and $\sigma_{\varepsilon}^2 = 1$.}\label{fig:simulated_SARspARCH}
\end{figure}

\begin{figure}
  \begin{center}
  \includegraphics[width=0.4\textwidth,natwidth=500,natheight=500]{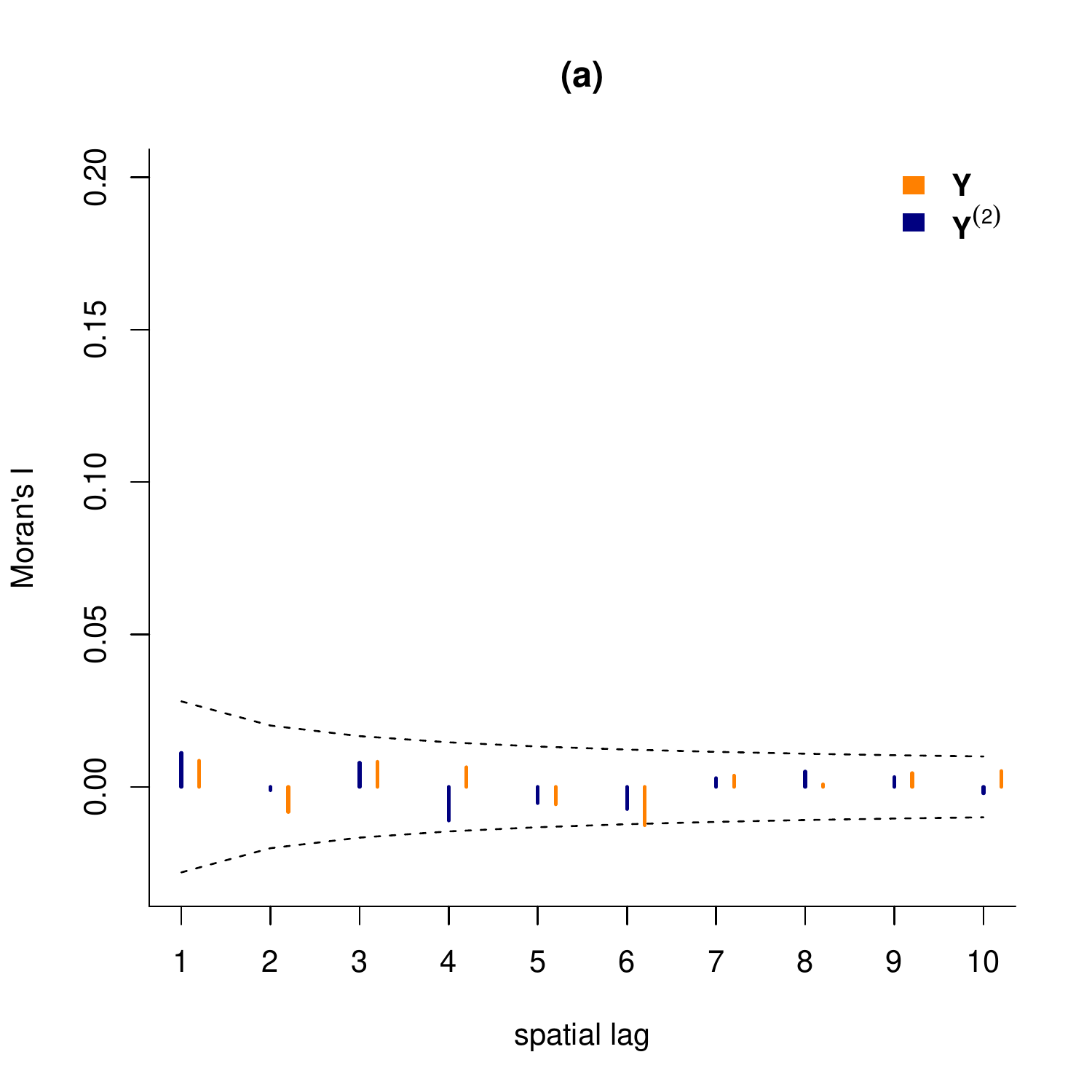}
  \includegraphics[width=0.4\textwidth,natwidth=500,natheight=500]{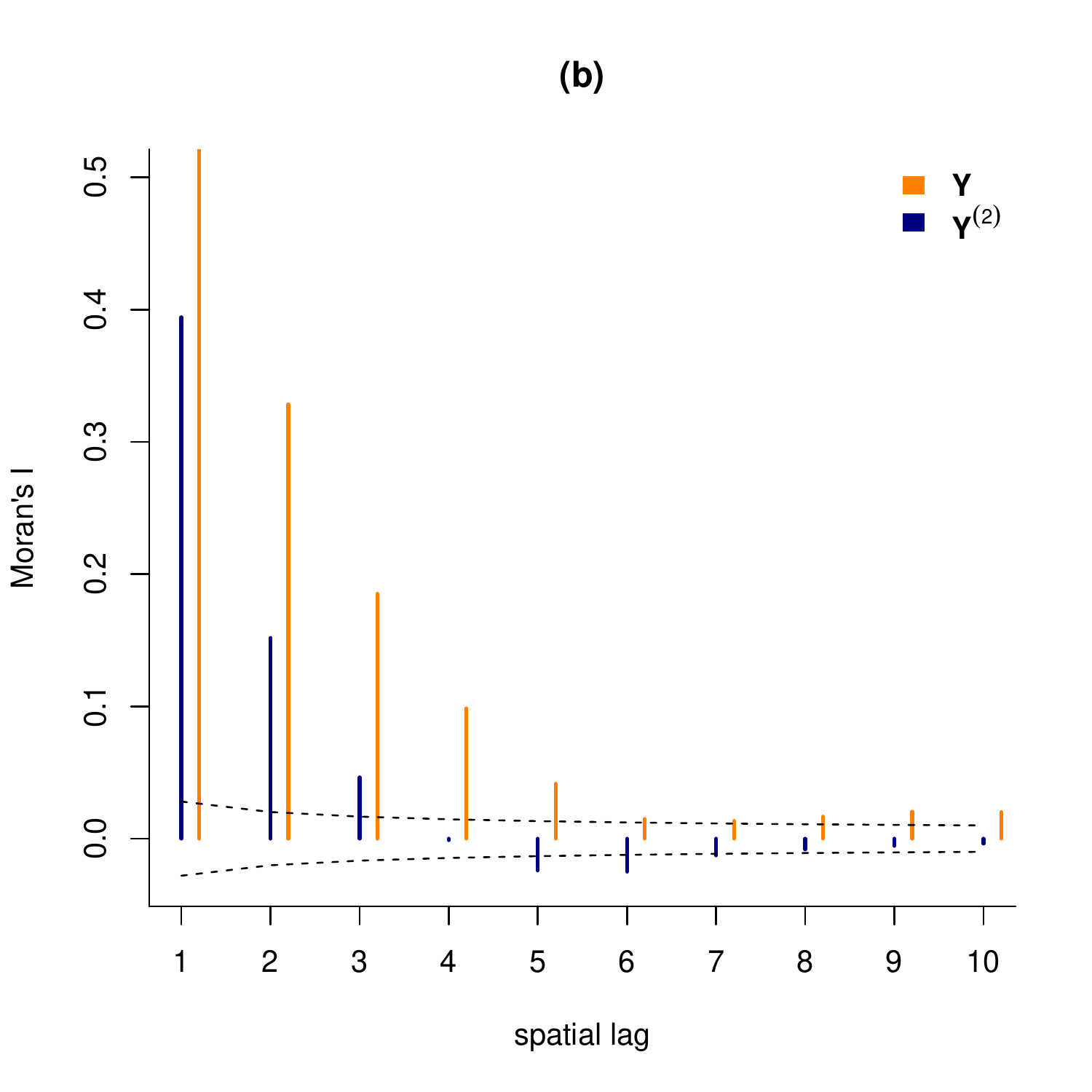}\\
  \includegraphics[width=0.4\textwidth,natwidth=500,natheight=500]{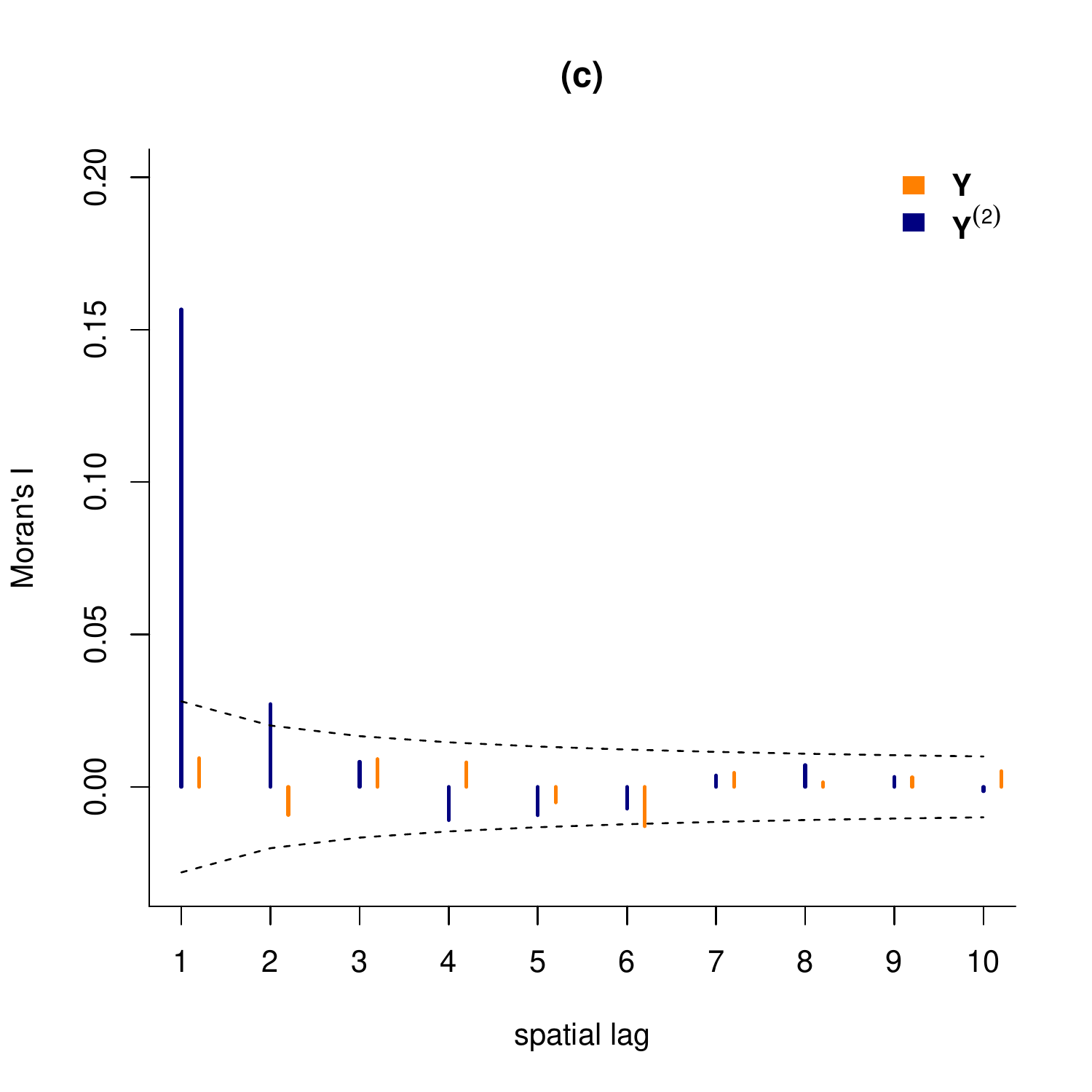}
  \includegraphics[width=0.4\textwidth,natwidth=500,natheight=500]{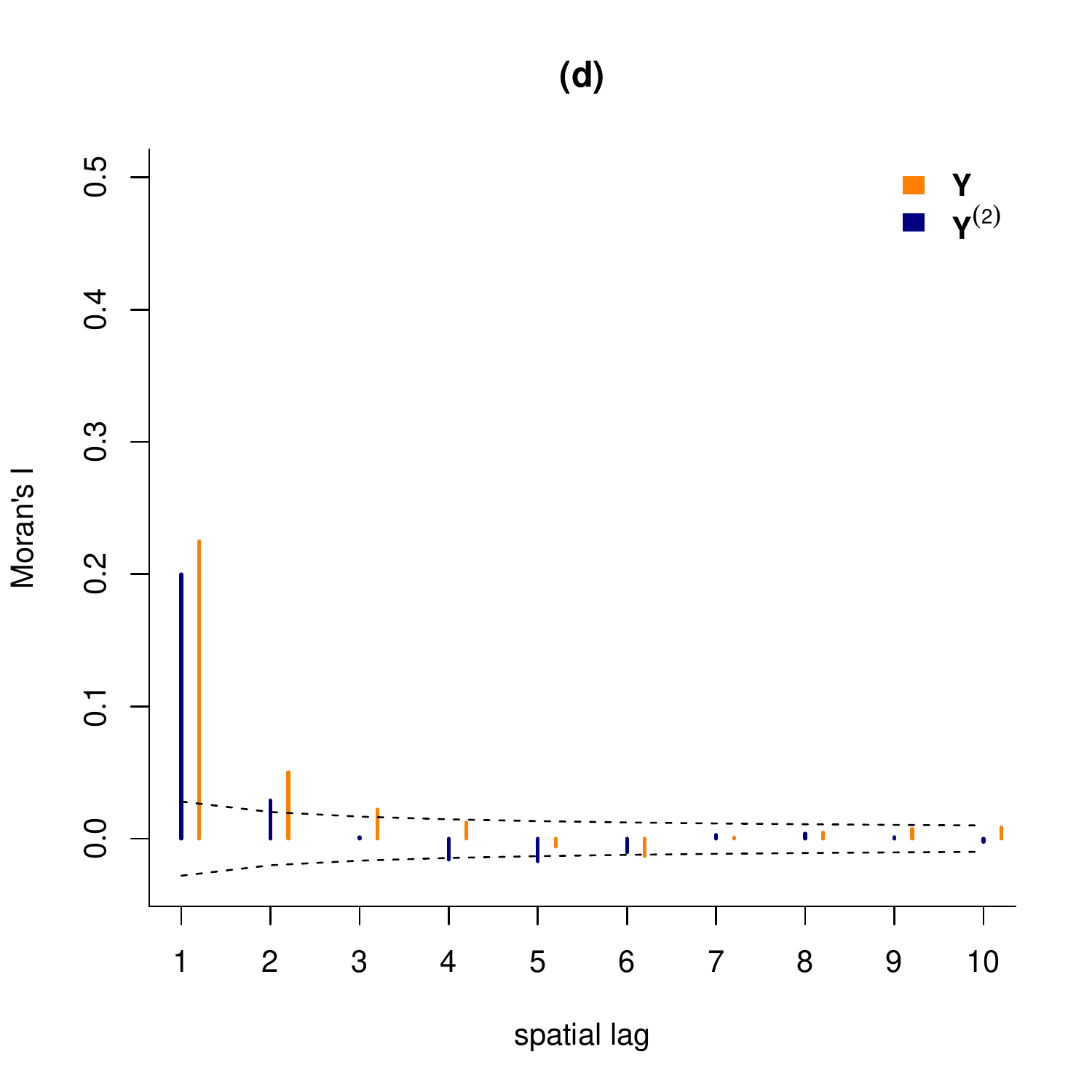}
  \end{center}
  \caption{Spatial autocorrelation function of the simulated spatial white noise process (a) truncated on $[-a,a]$, spatial autoregressive process (b), spatial ARCH process (c), spatial autoregressive process with spatial ARCH errors (d) plotted in Figure \ref{fig:simulated_SARspARCH}}\label{fig:ACF_simulated_SARspARCH}
\end{figure}

Finally, we briefly discuss the spatial autocorrelation function of the spatial ARCH process. In Figure \ref{fig:simulated_spARCH_ACF}, we plot a simulation of an oriented spatial ARCH process; i.e., the weighting matrix is triangular. Moreover, the spatial autocorrelation function is plotted for the observations $\xvec{Y}$ and the squared observations $\xvec{Y}^{(2)}$. More precisely, the autocorrelation function reports Moran's $I$ for different spatial lags; i.e., the first-order spatial lag consists of the directly neighboring locations, the second-order lag are all neighbors of these first-lag neighbors, and so forth. As expected, the observations $\xvec{Y}$ are not spatially autocorrelated, whereas the squared observations $\xvec{Y}^{(2)}$ exhibit a positive autocorrelation, which decreases with increasing order of the spatial lag.

\begin{figure}
  \begin{center}
  \includegraphics[width=0.4\textwidth,natwidth=500,natheight=500,trim = 0cm 0cm 0cm 4cm, clip = true]{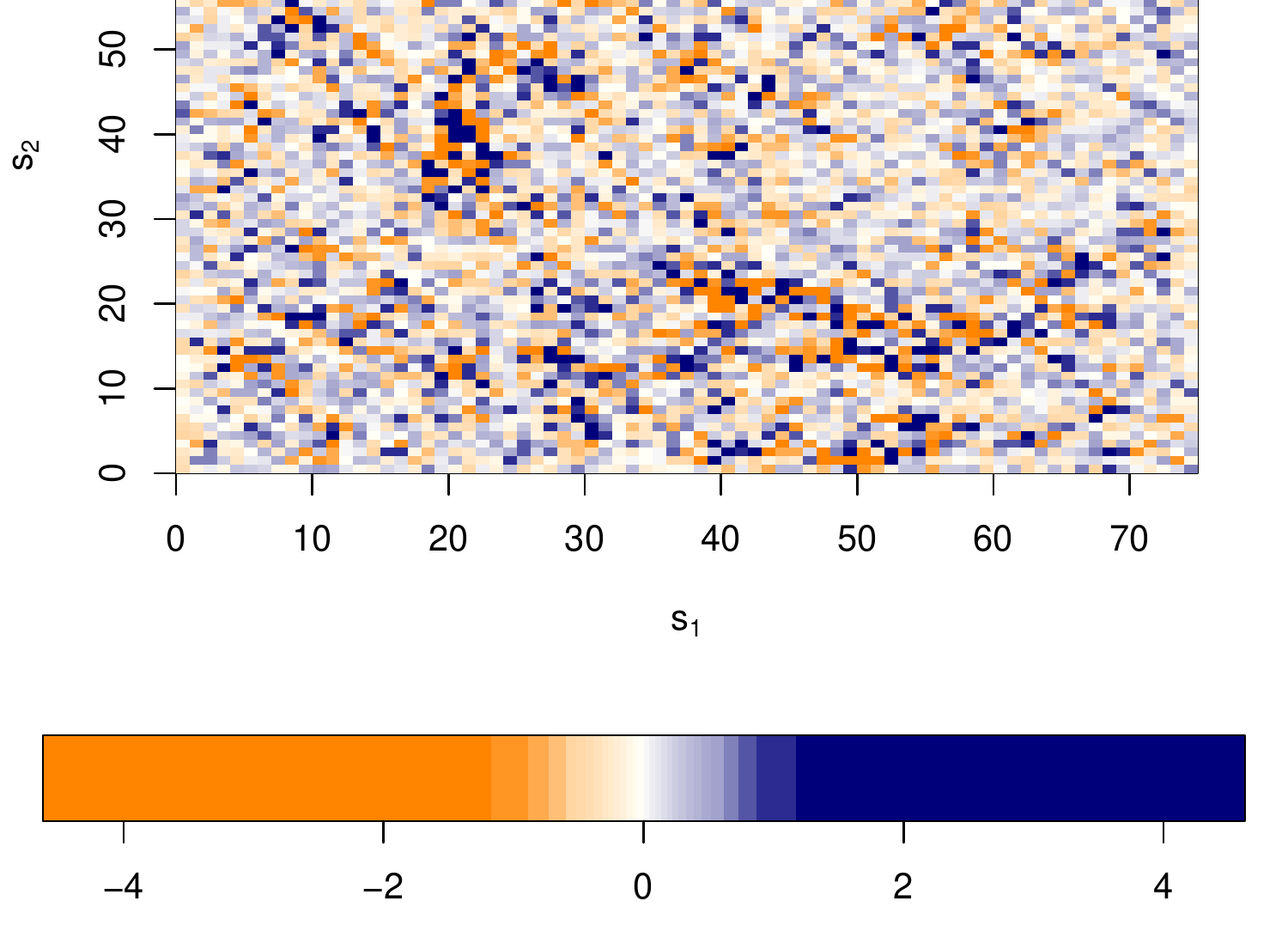}
  \includegraphics[width=0.4\textwidth,natwidth=1000,natheight=1000]{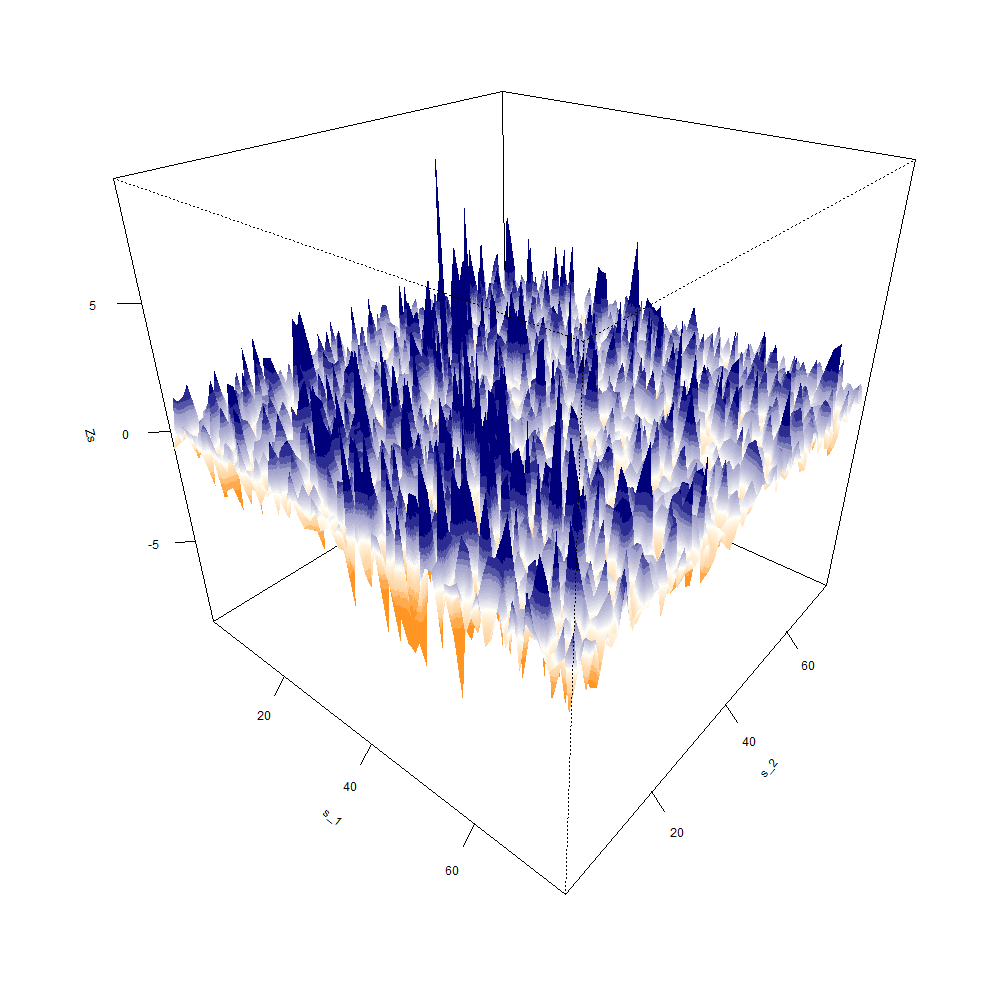} \\
  \includegraphics[width=0.55\textwidth,natwidth=500,natheight=500]{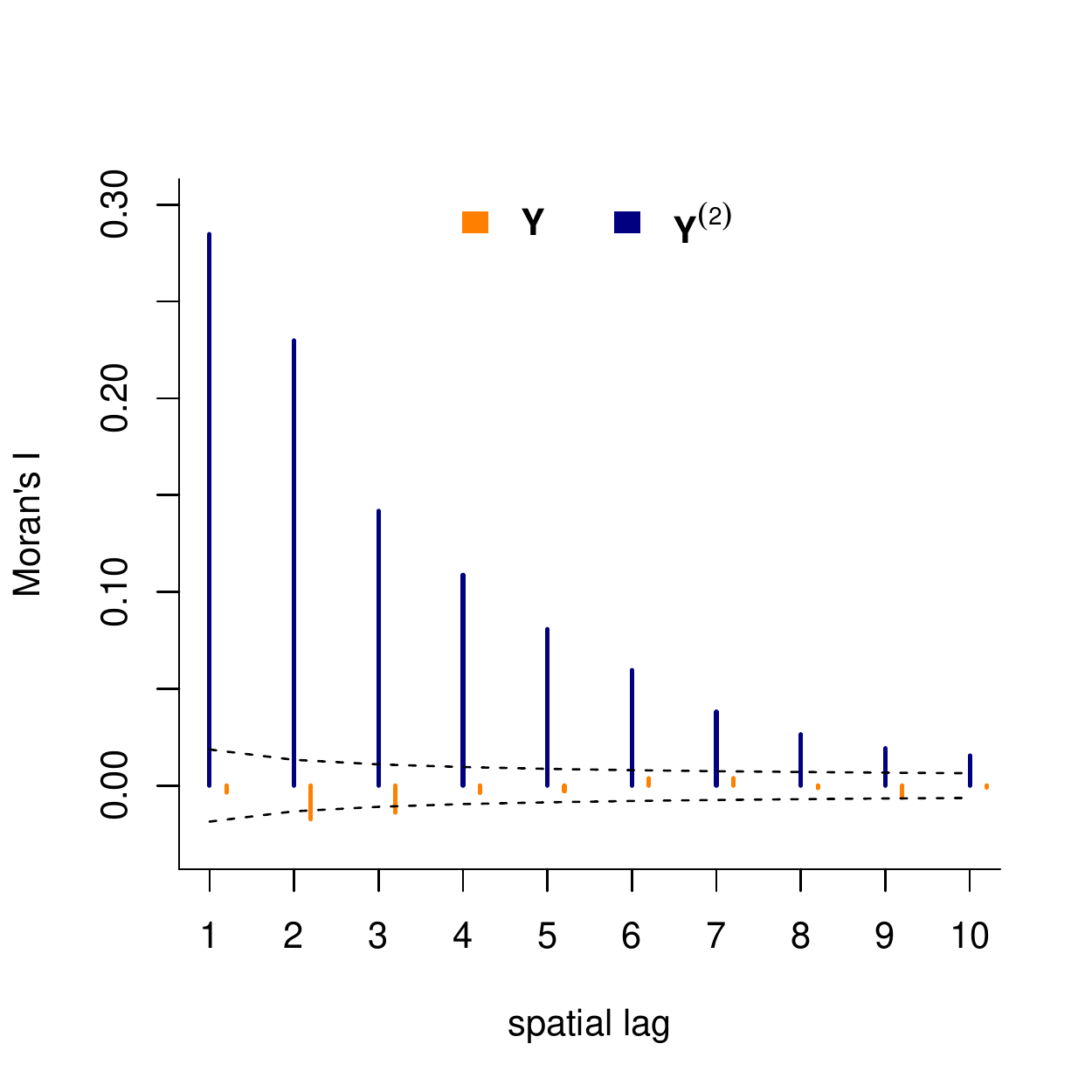}
  \end{center}
  \caption{Simulated oriented spatial ARCH process in the two- and three-dimensional view (above) and the spatial autocorrelation function of the simulated observations and squared observations (below).}\label{fig:simulated_spARCH_ACF}
\end{figure}

\subsection{Real Data Example: Cancer Mortality Rates}\label{sec:realdata}

In this section, we illustrate the proposed process using an empirical example. For this reason, we analyze the 5-year average mortality (2008--2012) caused by cancer of the lungs or bronchus provided by the Center for Disease Control and Prevention (\cite{CDC_data}). The death rates are age-adjusted to the 2000 U.S. standard population (cf. CDC (2015)). The spatial domain is all U.S. counties excluding Alaska and Hawaii, i.e., 3108 counties. Moreover, we do not distinguish in terms of race, sex, and age. In Figure \ref{fig:empirical_overview}, we show the mortality for lung cancer and the main covariates: particulate matter $\text{PM}_{2.5}$, the percentage of smokers in 2012, and the personal income per capita. In addition to these regressors, we include the amounts of nitrogen dioxide ($\text{NO}_{2}$), sulfate dioxide ($\text{SO}_{2}$), particulate matter $\text{PM}_{10}$, carbon monoxide ($\text{CO}$), and ozone ($\text{O}_{3}$) as regressors. Many studies have demonstrated that particulate matters are carcinogenic (cf., \cite{Raaschou13}, \cite{Cohen95}). Conversely, there is no association between traffic intensity, which results in a high amount of nitrogen dioxide, and the risk of lung cancer, as \cite{Raaschou13} noted.

All environmental data are annual averages (2012) recorded at the ground level by the United States Environmental Protection Agency (EPA). The measurement stations are plotted in the respective maps in Figure \ref{fig:empirical_overview}. Moreover, the data used as regressors are computed by spatial interpolation, in particular, inverse-distance-based kriging. Finally, we include covariates describing the health and economic status in each county, namely, the percentage of smokers in 2012 and the personal income per capita recorded by the CDC (Chronic Disease and Health Promotion Data \& Indicators) and the U.S. Department of Commerce, Bureau of Economic Analysis, respectively. The environmental covariates and the percentage of smokers are included in our analysis because they are the main drivers that cause a higher risk of lung cancer. Moreover, we include personal income to adjust for possible effects, such as better access to health care, early diagnosis/recognition, and screening.

\begin{figure}
  \begin{center}
  \includegraphics[width=0.4\textwidth,natwidth=1000,natheight=1000]{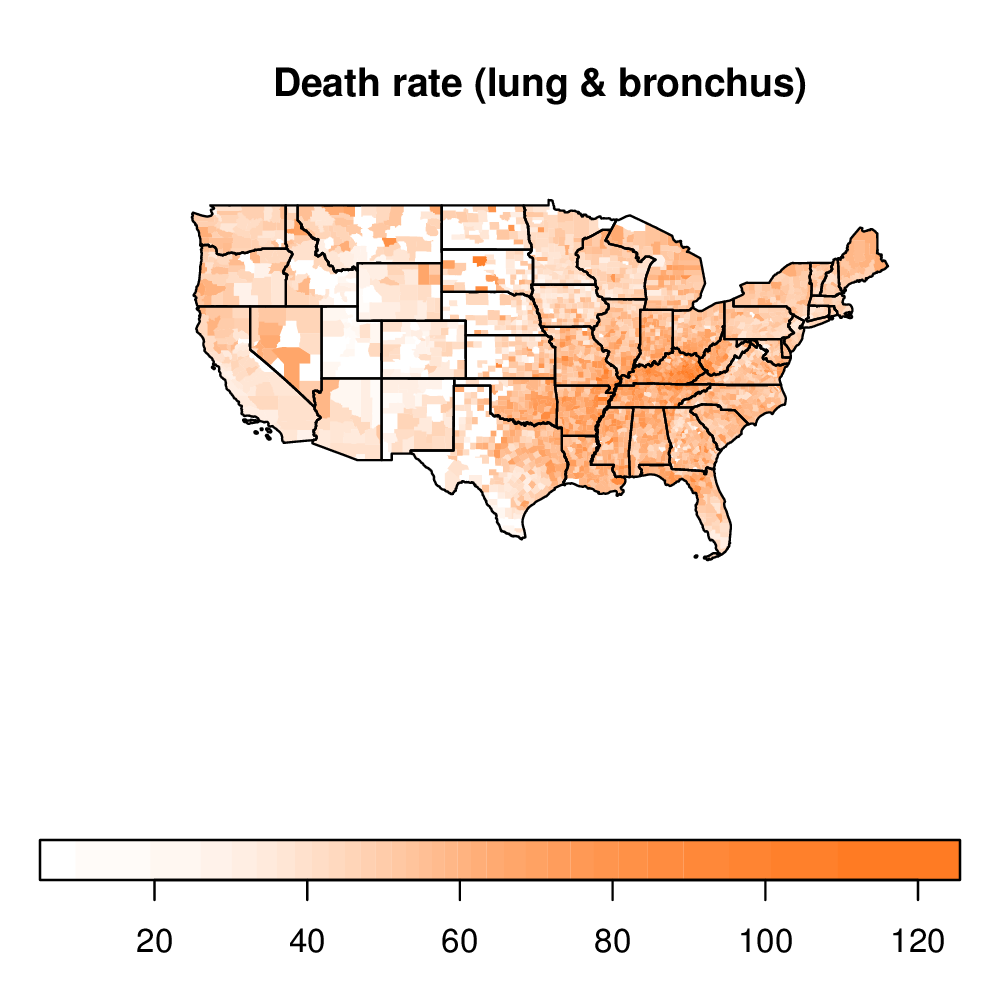}
  \includegraphics[width=0.4\textwidth,natwidth=1000,natheight=1000]{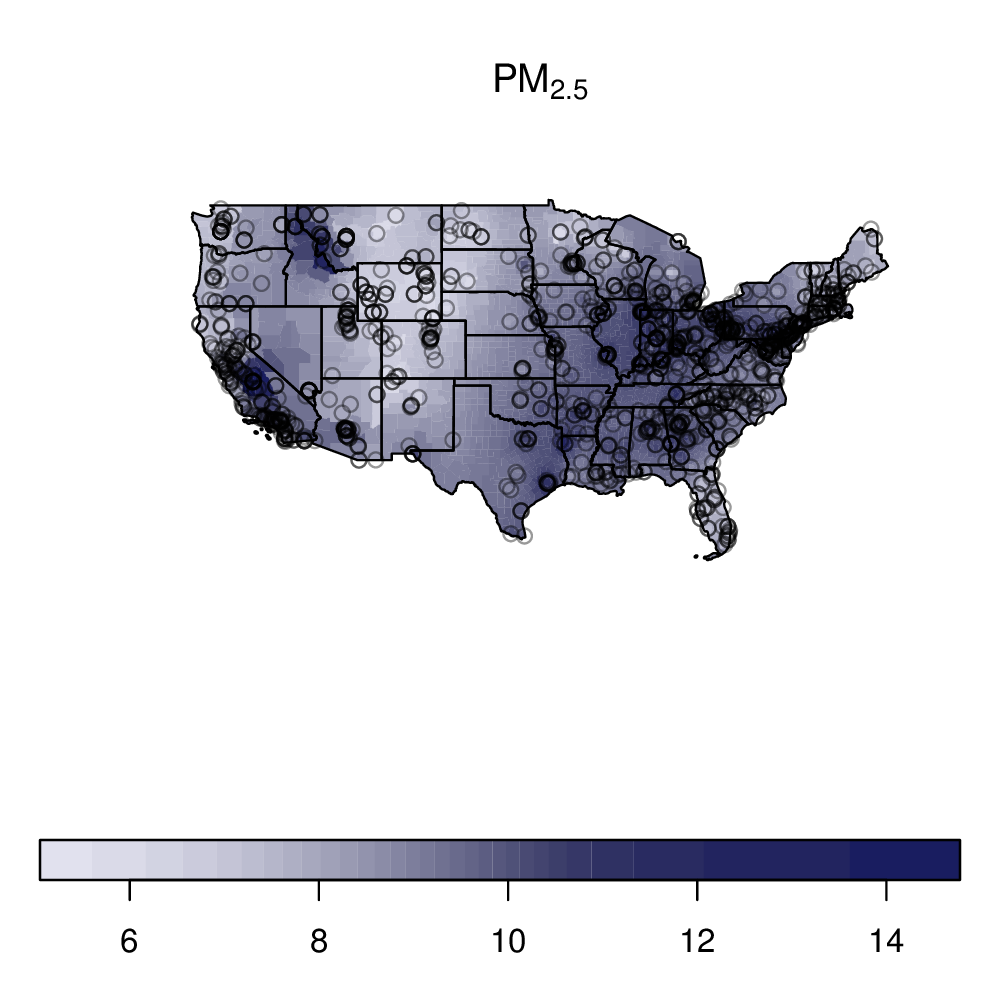}\\
  \includegraphics[width=0.4\textwidth,natwidth=1000,natheight=1000]{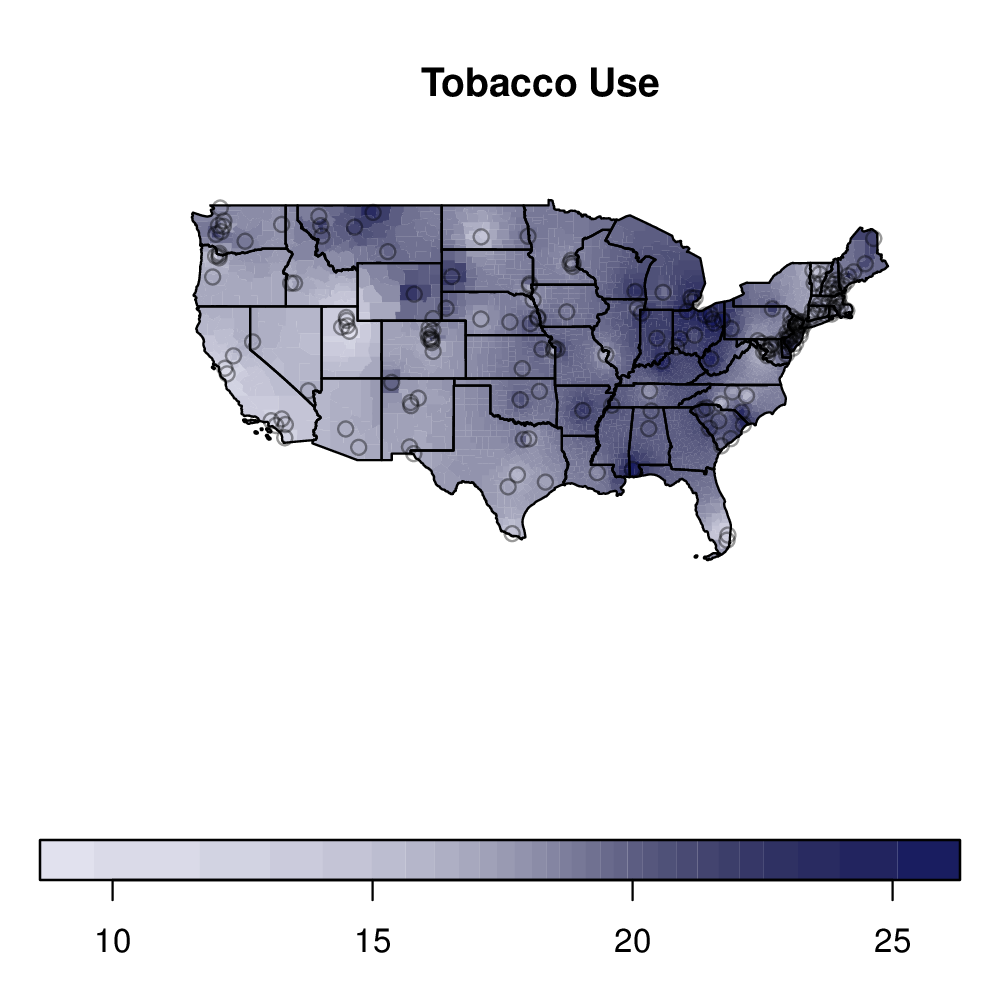}
  \includegraphics[width=0.4\textwidth,natwidth=1000,natheight=1000]{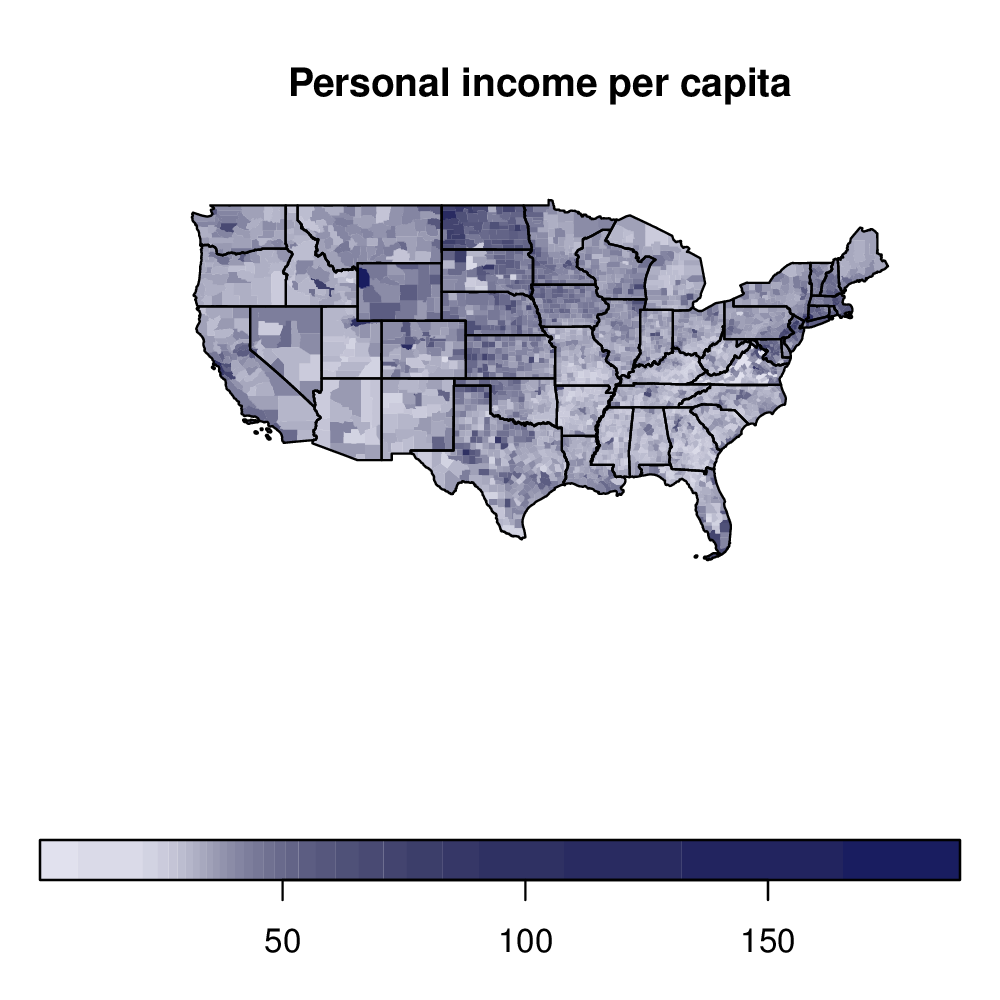}
  \end{center}
  \caption{Mortality caused by cancer of the lungs or bronchus in U.S. counties (above left) and main covariates: annual average of $\text{PM}_{2.5}$ in 2012 (above right), percentage of smokers in 2012 (below left), and personal income per capita in thousand U.S. dollars (below right). The measurement stations of the covariates $\text{PM}_{2.5}$ and the percentage of smokers are indicated on the maps via empty circles.}\label{fig:empirical_overview}
\end{figure}

The model given by \eqref{eq:SARspARCH1} and \eqref{eq:SARspARCH2} is estimated using the maximum likelihood approach. To include the covariates, the intercept $\mu\xvec{1}$ is replaced by $\xmat{X}\xvec{\beta}$, where $\xmat{X}$ is the matrix of regressors with the first column $\xvec{1}$. Because of the specific setting, we additionally incorporate two matrices of spatial weights for the autoregressive part; i.e., the model equation is given by
\begin{eqnarray*}
    \xvec{Y}   & = & \xmat{X}\xvec{\beta} + \left(\lambda_1 \xmat{B}_1 + \lambda_2 \xmat{B}_2\right) \xvec{Y} + \xvec{\xi} \\
    \xvec{\xi} & = & \text{diag}(\xvec{h})^{1/2} \xvec{\varepsilon} \qquad \text{with} \\
    \xvec{h}   & = & \alpha\xvec{1} + \rho\tilde{\xmat{W}}  \, \text{diag}(\xvec{\xi})\xvec{\xi} \, .\label{eq:SARspARCH_empirical}
\end{eqnarray*}
In particular, we estimate the spatial autoregressive part of the model using the well-known quasi-maximum-likelihood estimator with Gaussian errors $\xvec{\xi}$ (e.g., \citealt{Lee04}). Moreover, the spARCH parameters are included in the logarithmic likelihood function (cf. Section \ref{sec:inference}); i.e., all parameters are estimated in one step. The spatial weighting matrix of the spARCH process is chosen to be a non-triangular matrix $\xmat{W} = \rho \tilde{\xmat{W}}$. In particular, $\tilde{\xmat{W}}$ is defined as row-standardized Queen's contiguity matrix for all spatial lags up to order 5; i.e.,
\begin{equation*}
\tilde{\xmat{W}} = \text{diag}\left(\left(\sum_{k=1}^{5} \xmat{B}_k\right)\xvec{1}_n\right)^{-1} \left( \sum_{k=1}^{5} \xmat{B}_k \right)
\end{equation*}
with $\xmat{B}_k$ denoting the row-standardized binary contiguity matrix of the $k$-th-lag neighbors.
Thus, the weighting matrix $\xmat{B}_1$ is a classical row-standardized Queen's contiguity matrix of the first-lag neighbors, and matrix $\xmat{B}_2$ is the row-standardized contiguity matrix of the second-lag neighbors.

In Table \ref{table:empirical1}, we summarize the results of three models: a simple linear regression model, the SAR model, and the SAR\-sp\-ARCH model. Moreover, we report Moran's $I$ statistics and the $p$-values for testing the null hypothesis of the absence of spatial autocorrelation. All variables are log-transformed; thus, the estimates must be interpreted in elasticity terms, keeping the positive spatial correlation in mind (cf. \citealt{Lesage08}). We selected the regressors by minimizing the Akaike information criterion.
It is unsurprising that the covariate describing the behavioral aspect, namely, the percentage of smokers, has a large, positive impact on the mortality caused by lung cancer. Moreover, we observe only positive effects of the amount of nitrogen dioxide and $\text{PM}_{2.5}$ regarding the environmental covariates. However, it is important to distinguish between cancer incidence and cancer mortality. Hence, it is not surprising that we found different effects in terms of cancer mortality compared with the results of \cite{Raaschou13} and \cite{Cohen95}.

In all, it is interesting to compare the results of the linear regression model and the models that account for spatial dependence. All estimated parameters of the regression model are larger in absolute values than the estimated coefficients of the SAR model. For the SAR\-sp\-ARCH model, the coefficients are again smaller in absolute terms (e.g., percentage of smokers and all environmental effects), and several coefficients are omitted due to the Akaike information criterion (e.g., nitrogen dioxide, ozone). Hence, the spatial autocorrelation of the residual's variance also affects the results of the estimated coefficients and, therefore, the interpretation of the impact of the regressors. Thus, it would be interesting to analyze the impact of spatial heteroscedasticity on the estimated coefficients of an SAR model in more detail in future studies.

Moreover, the spatial autocorrelation of the residuals and the squared residuals are worth noting. Whereas Moran's $I$ of the residuals does not differ significantly from zero for both the SAR and SARspARCH models, the squared residuals are positively correlated for the SAR model. Consequently, the residual's variance exhibits spatial clusters, and the residuals cannot result from a spatial white noise process. However, by applying the proposed spARCH model to the residuals, it is possible to remove the spatial autocorrelation of the squared residuals. For the SARspARCH model, neither the residuals nor the squared residuals are correlated. 

\begin{landscape}
\begin{table}
  \caption{Estimated coefficients and summary statistics of a simple regression model as a benchmark and of the SAR and SARspARCH models for the mortality caused by lung cancer.}\label{table:empirical1}
  \begin{center}
\renewcommand{\baselinestretch}{1.0}
\large \normalsize
\begin{scriptsize}
  \begin{tabular}{l ccc c ccc c ccc}
  \hline
  \hline
  & \multicolumn{3}{c}{Linear Regression} && \multicolumn{3}{c}{SAR} && \multicolumn{3}{c}{SARspARCH} \\
  & Estimate & Standard Error & $p$-Value && Estimate & Standard Error & $p$-Value && Estimate & Standard Error & $p$-Value \\
  \hline
  $\quad$Intercept                               &-10.8575 & 0.9964 & 0.0000 && -4.3157 & 0.9347 & 0.0000 && -0.0059 & 0.1629 & 0.9712 \\
  \emph{Environmental}                           &         &        &        &&         &        &        &&         &        &        \\
  \emph{Covariates}                              &         &        &        &&         &        &        &&         &        &        \\
  $\quad$$\text{PM}_{10}$                        & -1.1734 & 0.0874 & 0.0000 && -0.3584 & 0.0844 & 0.0000 && -0.2641 & 0.0397 & 0.0000 \\
  $\quad$$\text{PM}_{2.5}$                       &  2.1193 & 0.1427 & 0.0000 &&  0.6162 & 0.1402 & 0.0000 &&  0.5365 & 0.0632 & 0.0000 \\
  $\quad$$\text{SO}_{2}$                         &  0.1210 & 0.0422 & 0.0042 &&  -      & -      & -      &&  -      & -      & -      \\
  $\quad$$\text{NO}_{2}$                         &  0.6217 & 0.0799 & 0.0000 &&  0.2731 & 0.0732 & 0.0002 &&  -      & -      & -      \\
  $\quad$$\text{O}_{3}$                          & -2.4133 & 0.2489 & 0.0000 && -0.8082 & 0.2251 & 0.0003 &&  -      & -      & -      \\
  $\quad$$\text{CO}$                             & -0.4041 & 0.1151 & 0.0005 && -0.1759 & 0.1026 & 0.0863 &&  -      & -      & -      \\
  \emph{Behavioral Covariates}                   &         &        &        &&         &        &        &&         &        &        \\
  $\quad$Tobacco Use                             &  1.2859 & 0.1677 & 0.0000 &&  0.6090 & 0.1381 & 0.0000 &&  0.3188 & 0.0665 & 0.0000 \\
  \emph{Economic Covariates}                     &         &        &        &&         &        &        &&         &        &        \\
  $\quad$Personal Income                         &  -      & -      & -      && -       & -      & -      &&  -      & -      & -      \\
  \emph{Spatial Coefficients}                    &         &        &        &&         &        &        &&         &        &        \\
  $\quad$$\lambda_1$                             &         &        &        &&  0.2449 & 0.0262 & 0.0000 &&  0.2624 & 0.0278 & 0.0000 \\
  $\quad$$\lambda_2$                             &         &        &        &&  0.4431 & 0.0327 & 0.0000 &&  0.3888 & 0.0400 & 0.0000 \\
  $\quad$$\sigma_{\xvec{\xi}}^2$                 &         &        &        &&  0.4628 & 0.0119 & 0.0000 &&         &        &        \\
  $\quad$$\alpha$                                &         &        &        &&         &        &        &&  0.0601 & 0.0015 & 0.0000 \\
  $\quad$$\rho$                                  &         &        &        &&         &        &        &&  0.6680 & 0.0161 & 0.0000 \\
  \emph{Summary Statistics}                      &         &        &        &&         &        &        &&         &        &        \\
  $\quad$Moran's $I$ $\xvec{\xi}$                &  0.2203 & 0.0106 & 0.0000 && -0.0114 & 0.0106 & 0.2966 &&         &        &        \\
  $\quad$Moran's $I$ $\xvec{\xi}^{(2)}$          &  0.3331 & 0.0106 & 0.0000 &&  0.3212 & 0.0106 & 0.0000 &&         &        &        \\
  $\quad$Moran's $I$ $\xvec{\varepsilon}$        &         &        &        &&         &        &        &&  0.0075 & 0.0106 & 0.4565 \\
  $\quad$Moran's $I$ $\xvec{\varepsilon}^{(2)}$  &         &        &        &&         &        &        &&  0.0067 & 0.0106 & 0.4852 \\
  $\quad$AIC                                     &         & 7091.062 &      &&         & 6560.509 &      &&         & 2484.686 &      \\
  \hline
  \end{tabular}
\end{scriptsize}
\renewcommand{\baselinestretch}{1.4}
\large \normalsize
  \end{center}
\end{table}
\end{landscape}

%
%
%
%
%
%


\section{Simulation Studies}\label{sec:MC}

The following section focuses on insights that we gained via extensive Monte Carlo simulation studies. Initially, we analyze the impact of the bounded support of the error distribution for the case of a non-triangular weighting matrix. Furthermore, we demonstrate how the parameters of the suggested spatial ARCH model can be estimated and illustrate the behavior of the estimators for finite samples. 

For all Monte Carlo simulations, we simulated the process as a two-dimensional lattice process; i.e.,
$D_{\xvec{s}} = \{\xvec{s} = (s_1, s_2)' \in \xset{Z}^2 : 0 \leq s_1,s_2 \leq d \}$. Hence, the number of observations $n$ is equal to $d^2$. Moreover, all simulations are performed for $10^5$ replications.

For the first simulation study, we use a common row-standardized Rook contiguity matrix. Consequently, the weighting matrix $\tilde{\xmat{W}}$ is set equal to the row-standardized Rook contiguity matrix $\xmat{R}_1 = (r_{1,ij})_{i,j = 1, \ldots, n}$, with
\begin{equation*}
r_{1,ij} = \left\{
\begin{array}{cc}
1 & \text{if} \qquad ||\xvec{s}_i - \xvec{s}_j||_{1} = 1\\
0 & \text{otherwise}
\end{array} \right. \, .
\end{equation*}
Furthermore, we include the parameter $\rho$ such that the weighting matrix is given by $\xmat{W} = \rho \tilde{\xmat{W}}$. Hence, the matrix $\xmat{W}$ is not triangular; thus, the support of the error distribution must be compact. Therefore, the residuals are simulated from a standard normal distribution truncated on the interval $[-a,a]$. The parameter $\xvec{\alpha}$ is chosen to be $5\cdot\xvec{1}_n$. Eventually, we simulated the process for different values of $\rho$ and calculated Moran's $I$ statistic of the squared observations to measure the extent of the spatial autocorrelation of the conditional variance (cf. \citealt{Moran50}).

In Figure \ref{fig:MoransI}, we plot Moran's $I$ and the resulting asymptotic 95\% confidence intervals of $I$ for different values of $\rho$. Obviously, the support does not have to be constrained regarding $\rho = 0$. However, this support decreases with increasing values of $\rho$. If $\rho = 1$, the parameter $a$ is equal to $0.968$. Moreover, we observe that the growth rate of $I$ decreases with increasing spatial weight. This trend can be explained by the compact support of the residuals. Because there cannot be large innovations $\varepsilon(\xvec{s}_i)$ in absolute terms, there also cannot occur large spatial clusters of high or low variance.

\begin{figure}
  \begin{center}
  \includegraphics[width=0.6\textwidth,natwidth=500,natheight=500]{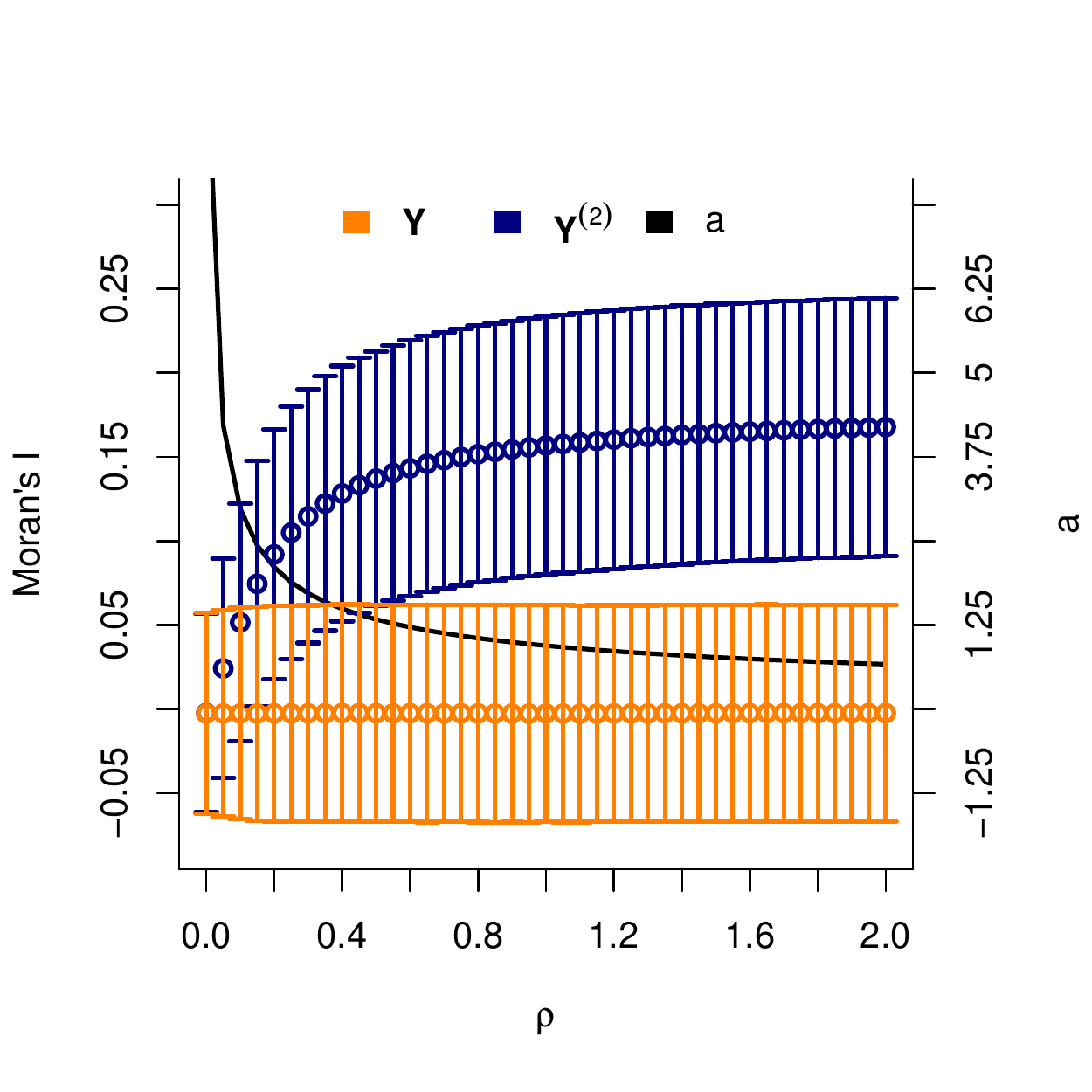}
  \end{center}
  \caption{Moran's $I$ of the observations $\xvec{Y}$ and the squared observations $\xvec{Y}^{(2)}$, including the asymptotic 95\% confidence intervals of $I$ for $\rho \in \{0,0.05, \ldots, 2\}$. Moreover, the resulting bound $a$ is plotted as a bold, black line.}\label{fig:MoransI} 
\end{figure}

Furthermore, we analyzed the performance of the proposed maximum-likelihood estimator in detail. For this simulation study, an oriented spatial ARCH process is considered; i.e., the weighting matrix $\xmat{W}$ is strictly triangular. We again utilize the matrix $\tilde{\xmat{W}}$, which results in a row-standardized binary weighting matrix $\xmat{R}_2 = (r_{2,ij})_{i,j = 1, \ldots, n}$ with
\begin{equation*}
r_{2,ij} = \left\{
\begin{array}{cc}
1 & \text{if} \qquad ||\xvec{s}_i - \xvec{s}_j||_{2} \leq \sqrt{2} \wedge  ||\xvec{s}_i - \xvec{s}_0||_{2} < ||\xvec{s}_j - \xvec{s}_0||_{2} \\
0 & \text{otherwise}
\end{array} \right.
\end{equation*}
and $\xvec{s}_0 = \left(\lfloor\frac{d}{2}\rfloor, \lfloor\frac{d}{2}\rfloor\right)'$. Consequently, any location $\xvec{s}_i$ is influenced by locations that lie within a distance of $\sqrt{2}$ from $\xvec{s}_i$ and that are closer to the origin $\xvec{s}_0$. The central location $\xvec{s}_0$ is chosen to be in the middle of the two-dimensional lattice $D_{\xvec{s}}$.

To evaluate the performance of the estimators, we consider the simple spARCH(1) model with
\begin{eqnarray*}
\xvec{Y} & = & \text{diag}(\xvec{h})^{1/2} \xvec{\varepsilon} \\
\xvec{h} & = & \alpha\xvec{1}_n + \rho\tilde{\xmat{W}}  \, \text{diag}(\xvec{Y})\xvec{Y} \, .
\end{eqnarray*}
In Figure \ref{fig:MC_spARCH}, the performance of the estimators for both parameters $\alpha$ and $\rho$ is visualized using kernel density estimates. The process was simulated for $d \in \{10, 20, 50\}$ with $10^5$ replications. Moreover, we considered all combinations of the true parameters $\alpha \in \{0.5, 1, 2, 5\}$ and $\rho \in \{0, 0.2, 0.6, 0.9\}$; i.e., the simulation study was performed for $48$ settings.

Regarding the first plot (Ia) in Figure \ref{fig:MC_spARCH}, one might see that the true parameter of $\alpha$ is slightly underestimated if $\rho = 0$, although the density of the estimated $\hat{\alpha}$ is sharper. For increasing values of $\rho$, the densities become less sharp, but the estimates are unbiased. Furthermore, we analyzed the performance of the estimators for an increasing number of observations and fixed $\rho = 0.2$ (see plot (IIb)). Interestingly, the smaller values of $\alpha$ are estimated more precisely than larger values of $\alpha$. However, all estimators seem to be unbiased and consistent. For the estimator of $\rho$, the performance does not depend on the magnitude of the spatial autocorrelation in the variance (see plot (IIb)). All density estimates are equally shaped. However, one might observe that the estimator works poorly if the number of observations and the parameter $\rho$ are small. For $\rho = 0.2$ and $d = 10$, $\hat{\rho}$ is more often close to zero than to the correct value of $0.2$. If either $\rho$ or $d$ is increasing, the bias vanishes. In all settings, the absence of dependence in the variance, i.e., $\rho = 0$, is estimated better than the presence of spatial clusters in the variance, i.e., $\rho > 0$. Moreover, the estimation of $\rho$ is independent of the coefficient $\alpha$. Please note that the curves in plot (IIa) are identical because the random seed was set to the number of replicates for each setting; i.e., the innovations $\xvec{\varepsilon}$ are identical for each setting.

\begin{figure}
  \begin{center}
\begin{tabular}
{>{\centering\arraybackslash}m{0.5cm}  >{\centering\arraybackslash}m{0.45\textwidth} >{\centering\arraybackslash}m{0.45\textwidth}}
& (I) & (II) \\
(a) & \includegraphics[width=0.45\textwidth,natwidth=500,natheight=500]{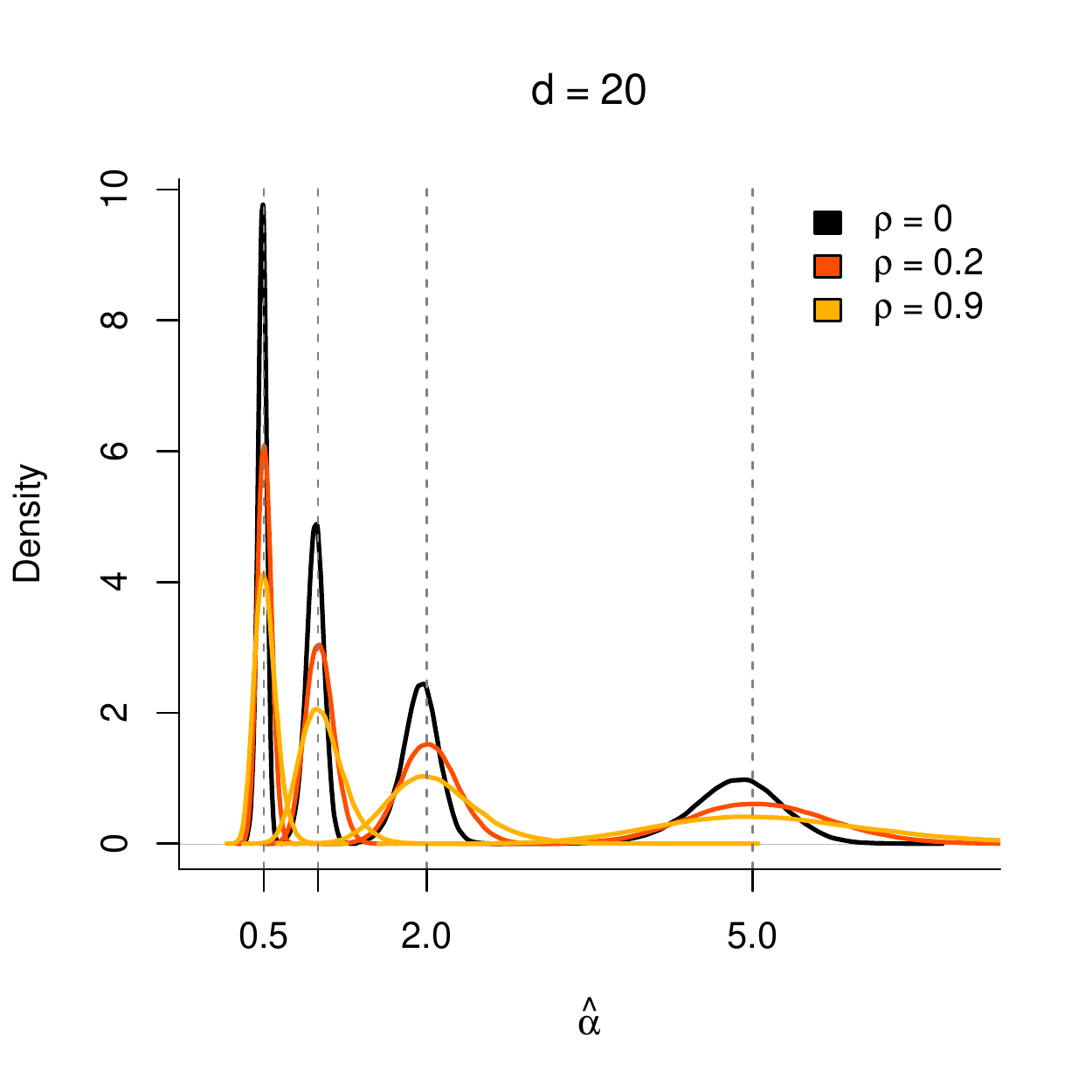} &
      \includegraphics[width=0.45\textwidth,natwidth=500,natheight=500]{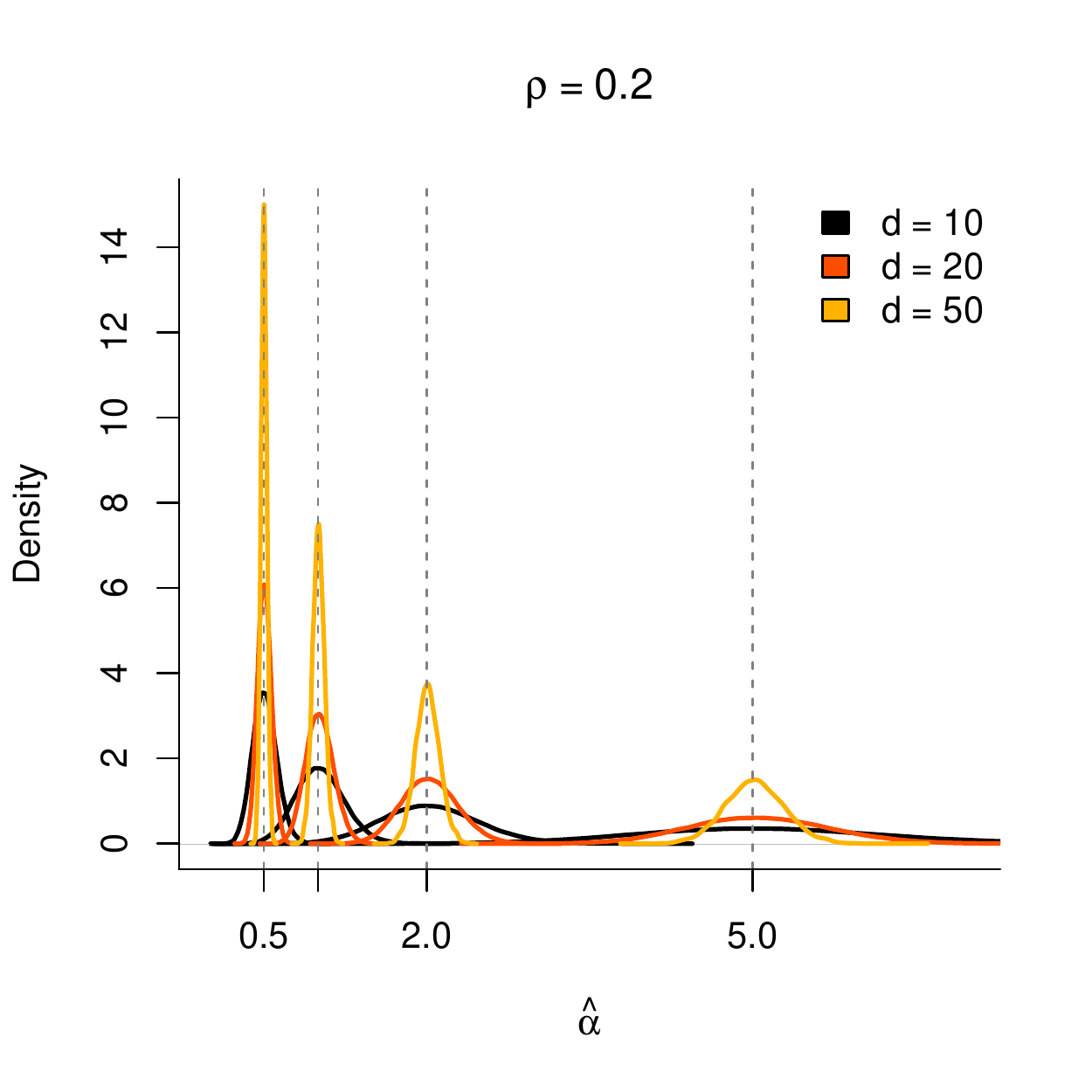} \\
(b) & \includegraphics[width=0.45\textwidth,natwidth=500,natheight=500]{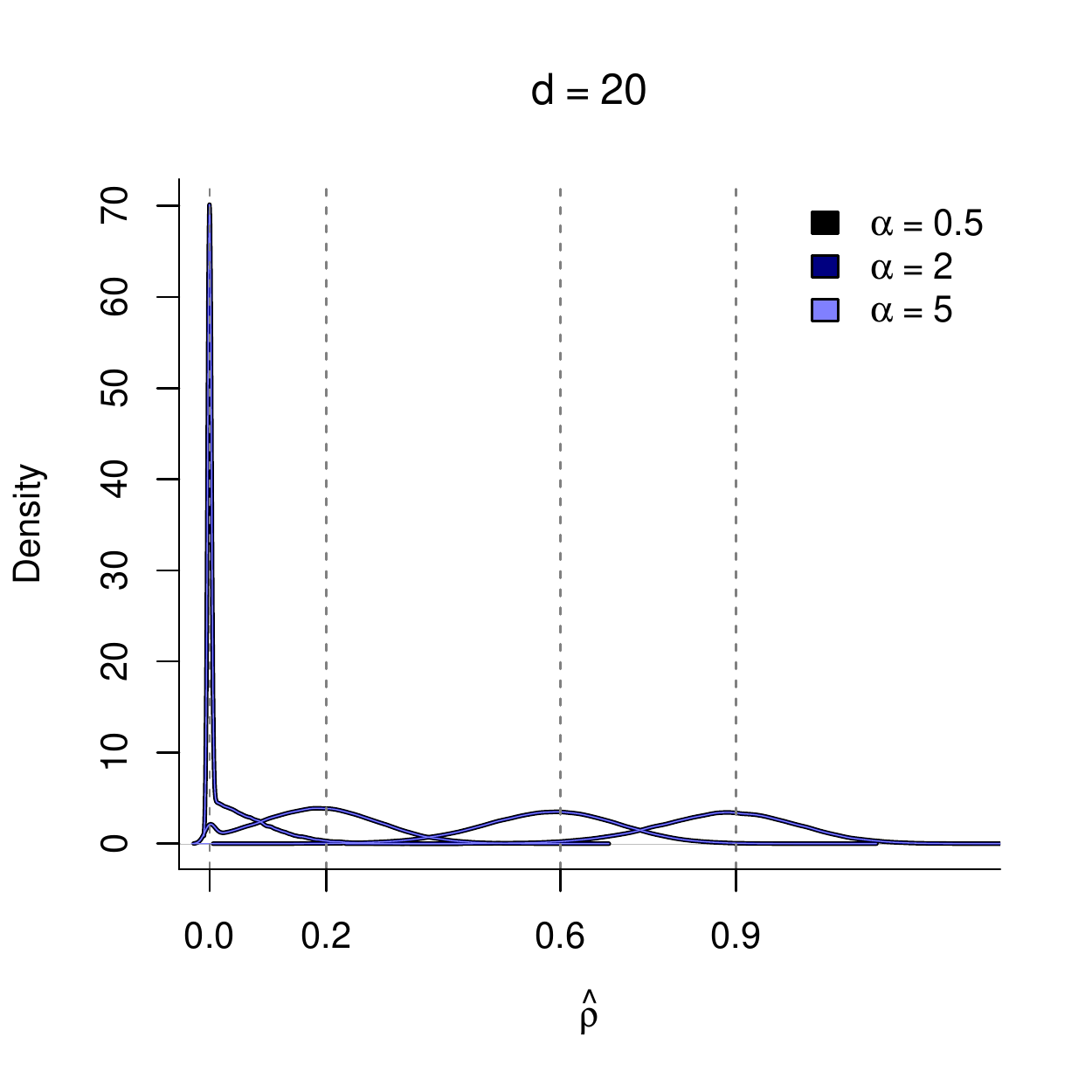} &
      \includegraphics[width=0.45\textwidth,natwidth=500,natheight=500]{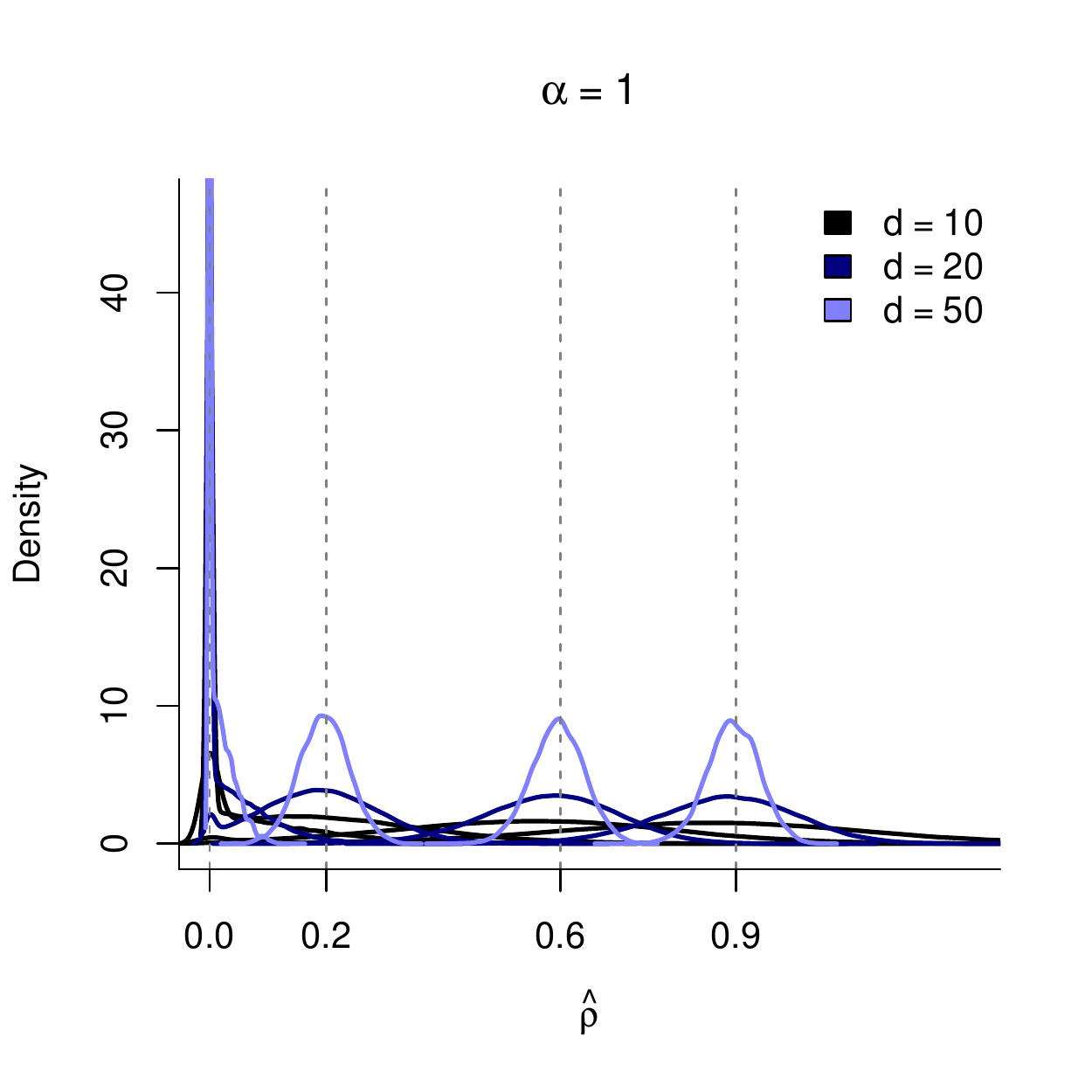}
\end{tabular}
  \end{center}
  \caption{Kernel density estimates of (a) $\hat\alpha$ and (b) $\hat\rho$ for (I) a constant number of observations ($d=20$) and (II) an increasing number of observations. The true values of the parameters are  $\alpha \in \{0.5, 1, 2, 5\}$ and $\rho \in \{0, 0.2, 0.6, 0.9\}$ for the different settings.}\label{fig:MC_spARCH}
\end{figure}

\section{Discussion}\label{sec:conclusion}

Finally, we discuss possible extensions of the model in this section and conclude the paper by summarizing the main findings. One possible extension of the proposed spARCH process would be to consider a generalized version analogous to the GARCH process introduced by \cite{Bollerslev86}. For the spatial ARCH process, we defined the conditional spatial variance by \eqref{eq:spARCH}; i.e.,
\begin{equation*}
\xvec{h} = \xvec{\alpha} + \xmat{W}_1  \, \text{diag}(\xvec{Y})\xvec{Y} \, .
\end{equation*}
Adding a weighting matrix $\xmat{W}_2$ for $\xvec{h}$ leads to
\begin{equation*}
\xvec{h} = \xvec{\alpha} + \xmat{W}_1  \, \text{diag}(\xvec{Y})\xvec{Y} + \xmat{W}_2 \xvec{h} \,,
\end{equation*}
which is equivalent to
\begin{equation*}
\xvec{h} = (\xmat{I} - \xmat{W}_2)^{-1} \left(\xvec{\alpha} + \xmat{W}_1  \, \text{diag}(\xvec{Y})\xvec{Y}\right) \, .
\end{equation*}
The weighting matrix $\xmat{W}_2$ consists of the weights for the spatial moving average part, and it can be chosen analogous to matrix $\xmat{W}_1$. Surely, the matrix must be non-stochastic with zeros on the diagonal, and the determinant of $(\xmat{I} - \xmat{W}_2)$ must not be zero. For the abovementioned process, the $i$-th component of $\xvec{h}$ is given by
\begin{equation*}
h(\xvec{s}_i) = \alpha_i + \sum_{v=1}^{n} w_{1,iv} Y(\xvec{s}_v)^2 + \sum_{v=1}^{n} w_{2,iv} h(\xvec{s}_v) \, .
\end{equation*}
Consequently, this spatial GARCH process incorporates a spatial autoregressive and moving-average part in the conditional variance. However, the moments of this process are not straightforward; thus, this process should be considered in more detail in the future. A further possible extension would be a multivariate spatial process with conditional heteroscedasticity; i.e., we do not observe a univariate random variable at each location but rather a vector of observations.

For the introduced spatial ARCH process, we derived the required conditions such that the process is well defined. In particular, certain assumptions regarding the convergence of $\xmat{A}^{2k}$ are necessary if the weighting matrix is not triangular. Furthermore, we analyzed the moments of this new spatial model and proposed an estimation strategy based on the maximum likelihood approach. Via extensive simulation studies, the performance of this estimator is illustrated. To focus on empirical problems, we discussed possible spatial weighting schemes in detail.

Moreover, we introduced a spatial autoregressive process with heteroscedastic residuals (SARspARCH). In particular, we applied this process to the cancer death rate in all U.S. counties except Alaska and Hawaii. For this empirical example, we included environmental, economic, and health-behavioral covariates. Comparing the estimation results of a spatial autoregressive (SAR) and the proposed SARspARCH process, one might observe that the regression coefficients are slightly different. In particular, the effect implied by the number of smokers is underestimated if we do not account for heteroscedastic residuals. Whereas the estimated coefficient equals 0.61 for the SAR model, the estimate is 0.32 for the SARspARCH model. In the future, it would be interesting to analyze whether the estimators of an SAR process are biased if the variance of the residuals exhibit spatial clusters. Moreover, the sensitivity of our process and the introduced maximum likelihood estimator should be analyzed in more detail with respect to the choice of the weighting matrices. In particular, the focus should be on the assumption of an oriented process, i.e., in the case when the assumed weighting matrix is strictly triangular, although the process is not oriented. Moreover, the performance of the likelihood estimator of the parameters of an SAR model under spatial conditional heteroscedasticity should be critically examined, as we noted above.

\section*{Appendix}

\begin{appendix}


\section{Proofs}

\begin{proof}[Proof of Theorem \ref{th:f_eps}]
We observe that for $i \in \{1,\ldots,n\}$
\begin{eqnarray}
Y(\xvec{s}_i)^2 & = & \varepsilon(\xvec{s}_i)^2 h(\xvec{s}_i) \notag \\
                & = & \alpha_i \varepsilon(\xvec{s}_i)^2 + \varepsilon(\xvec{s}_i)^2 \sum_{v=1}^{n} w_{iv} \underbrace{Y(\xvec{s}_v)^2}_{=h(\xvec{s}_v)\varepsilon(\xvec{s}_v)^2} \notag \\
                & = & \alpha_i \varepsilon(\xvec{s}_i)^2 + \varepsilon(\xvec{s}_i)^2 \sum_{v=1}^{n} \alpha_v w_{iv} \varepsilon(\xvec{s}_v)^2 \notag \\
                &   & + \varepsilon(\xvec{s}_i)^2 \sum_{v=1}^{n} w_{iv} \varepsilon(\xvec{s}_v)^2 \sum_{\substack{j=1}}^{n} w_{vj} Y(\xvec{s}_j)^2 \, . \label{eq:proof_th1}
\end{eqnarray}
\eqref{eq:proof_th1} can be rewritten in matrix notation as follows:
\begin{equation*}
\xvec{\eta} = \left( \xmat{I} - \xmat{A}^2 \right) \xvec{Y}^{(2)} \, .
\end{equation*}
The system of linear equations has a unique solution if \eqref{eq:det} is fulfilled. Thus,\linebreak $\xvec{Y}(\xvec{s}_1)^2,\ldots,\xvec{Y}(\xvec{s}_n)^2$ are uniquely defined by $\xvec{\varepsilon}(\xvec{s}_1)^2,\ldots,\xvec{\varepsilon}(\xvec{s}_n)^2$. Because $\xvec{h} = \xvec{\alpha} + \xmat{W}(\xmat{I} - \xmat{A}^2)^{-1}\xvec{\eta}$ and $\xvec{Y} = \text{diag}(\xvec{h})^{1/2}\xvec{\varepsilon}$, the result follows.
\end{proof}

\begin{proof}[Proof of Theorem \ref{th:spARCH}]
The result is obvious because all elements of $\xvec{\eta}$ are nonnegative.
\end{proof}

\begin{proof}[Proof of Lemma \ref{lemma:triangular}]
If $\xmat{W}$ is a lower triangular matrix, it is nilpotent because $\xmat{W}^n=\xmat{0}$. The same holds for the matrix $\xmat{A}$; i.e., $\xmat{A}^n = \xmat{0}$. Because $rk(\xmat{I} - \xmat{A}^2)=n$  and $(\xmat{I} - \xmat{A})^{-1} = \xmat{I} + \xmat{A} + \ldots + \xmat{A}^{n-1}$, it follows that
\begin{equation}\label{eq:Inv}
(\xmat{I} - \xmat{A}^2)^{-1} = \xmat{I} + \xmat{A}^2 + \ldots + \xmat{A}^{2[(n-1)/2]} .
\end{equation}
All elements of $\xmat{A}$ are nonnegative; thus, the result follows straightforwardly.
\end{proof}

\begin{proof}[Proof of Lemma \ref{lemma:nontriangular}]
We make use of Theorem 18.2.16 of \cite{Harville97}. Thus, if \linebreak$\lim_{k\rightarrow\infty} \xmat{A}^{2k} = \xmat{0}$,
it follows that $\det\left( \xmat{I} - \xmat{A}^2 \right) \neq 0$. Consequently, Theorem \ref{th:f_eps} can be applied. Moreover, it holds that
\begin{equation*}
\left( \xmat{I} - \xmat{A}^2 \right)^{-1} = \sum_{v=0}^\infty (\xmat{A}^2)^v \, .
\end{equation*}
Because all elements of $\xmat{A}$ are nonnegative, it follows that all components of the matrix $\left( \xmat{I} - \xmat{A}^2 \right)^{-1}$ are also nonnegative.
\end{proof}

\begin{proof}[Proof of Theorem \ref{th:distr_sym}]
We utilize \eqref{eq:Y2} and obtain
\begin{equation*}
\xvec{h} = \xvec{\alpha} + \xmat{W} ( \xmat{I}_n - \xmat{A}^2 )^{-1} \xvec{\eta} = k(\varepsilon(\xvec{s}_1)^2,\ldots, \varepsilon(\xvec{s}_n)^2) \, .
\end{equation*}
Consequently,
\begin{eqnarray*}
\xvec{Y}^\prime
& = &
\text{diag}(k(\varepsilon(\xvec{s}_1)^2,\ldots, \varepsilon(\xvec{s}_n)^2)) \; (\varepsilon(\xvec{s}_1),\ldots,\varepsilon(\xvec{s}_n))^\prime \\
& \stackrel{d}{=} &
\text{diag}(k( ( (-1)^{v_1} \varepsilon(\xvec{s}_1))^2,\ldots, ( (-1)^{v_n} \varepsilon(\xvec{s}_n))^2)) \; 
( (-1)^{v_1} \varepsilon(\xvec{s}_1),\ldots, (-1)^{v_n} \varepsilon(\xvec{s}_n))^\prime \\
& = &
( (-1)^{v_1} Y(\xvec{s}_1),\ldots, (-1)^{v_n} Y(\xvec{s}_n))^\prime .
\end{eqnarray*}
Thus, the result is proved.
\end{proof}

\begin{proof}[Proof of Lemma \ref{lemma:moments1}]
First, let $||.||$ be an arbitrary induced matrix norm. Because
\begin{equation*}
||\xvec{Y}^{(2)}||^r \le ||(\xmat{I} - \xmat{A}^2)^{-1} ||^r \; ||\xvec{\eta}||^r \, ,
\end{equation*}
it follows that
\begin{equation*}
E( ||\xvec{Y}^{(2)}||^r ) \le \sqrt{ E( ||(\xmat{I} - \xmat{A}^2)^{-1} ||^{2r} ) \; E( ||\xvec{\eta}||^{2r} ) } .
\end{equation*}
In (\ref{eq:Inv}), it is shown that
\begin{equation*}
(\xmat{I} - \xmat{A}^2)^{-1} = \xmat{I} + \xmat{A}^2 + \ldots + \xmat{A}^{2[(n-1)/2]} .
\end{equation*}
Consequently,
\begin{eqnarray*}
||(\xmat{I} - \xmat{A}^2)^{-1}|| & \le & ||\xmat{I}|| + ||\xmat{A}^2|| + \ldots + ||\xmat{A}^{2[(n-1)/2]}|| \\
& \le & \sum_{v=0}^{[(n-1)/2]} ||\xmat{A}||^{2v} \\
& \le & \sum_{v=0}^{[(n-1)/2]} ||\mbox{diag}(\varepsilon(\xvec{s}_1)^2,\ldots,\varepsilon(\xvec{s}_n)^2)||^{2v} \; ||\xmat{W}||^{2v}
 \end{eqnarray*}
and by Jensen's inequality,
\begin{eqnarray*}
||(\xmat{I} - \xmat{A}^2)^{-1}||^{2r} & \le & \left( \left[\frac{n-1}{2}\right] + 1 \right)^{2r-1} \;
\sum_{v=0}^{[(n-1)/2]} ||\mbox{diag}(\varepsilon(\xvec{s}_1)^2,\ldots,\varepsilon(\xvec{s}_n)^2)||^{4rv} \; ||\xmat{W}||^{4rv} \, .
\end{eqnarray*}
This leads to
\begin{footnotesize}
\begin{eqnarray*}
E(||(\xmat{I} - \xmat{A}^2)^{-1}||^{2r}) & \le &
\left( \left[\frac{n-1}{2}\right] + 1 \right)^{2r-1} \;
\sum_{v=0}^{[(n-1)/2]}  \; ||\xmat{W}||^{4rv} \; E\left(
||\mbox{diag}(\varepsilon(\xvec{s}_1)^2,\ldots,\varepsilon(\xvec{s}_n)^2)||^{4rv} \right) .
\end{eqnarray*}
\end{footnotesize}
Taking the norm $||.||_1$, we obtain that
\[  E\left(  ||\mbox{diag}(\varepsilon(\xvec{s}_1)^2,\ldots,\varepsilon(\xvec{s}_n)^2)||_1^{4rv} \right)
= \max\limits_{1 \le i \le n} E( \varepsilon(\xvec{s}_i)^{8rv} ) . \]
This shows that for the existence of the upper bound, it is required that $E( \varepsilon(\xvec{s}_i)^{8r[(n-1)/2]} ) $ must exist. For the existence of $E(||\xvec{\eta}||^{2r})$, it is sufficient that $E( \varepsilon(\xvec{s}_i)^{4r} )$ exists.

Regarding b), one can see that $E(Y(\xvec{s}_i)^{2v-1}) = 0$ because the distribution is symmetric, and the moments exist.
$(Y(\xvec{s}_1),\ldots, Y(\xvec{s}_n))$ and $(-Y(\xvec{s}_1), Y(\xvec{s}_2),\ldots, Y(\xvec{s}_n))$ have the same distribution. Thus, it follows that
\begin{equation*}
E(Y(\xvec{s}_1)^{2v-1} | Y(\xvec{s}_2),\ldots, Y(\xvec{s}_n) )  = E(- Y(\xvec{s}_1)^{2v-1} | Y(\xvec{s}_2),\ldots, Y(\xvec{s}_n) ) \, .
\end{equation*}
Consequently, this quantity is equal to zero.
\end{proof}

\begin{proof}[Proof of Theorem \ref{th:moments}]
Now,
\begin{equation*}
|| \left(Y(\xvec{s}_1)^2, \ldots, Y(\xvec{s}_n)^2\right) || \; \leq \; ||(\xmat{I} - \xmat{A}^2)^{-1}|| \;
|| \xvec{\eta} || \; \leq \; \lambda \; ||\xvec{\eta}|| \, .
\end{equation*}
Choosing the norm $|| \cdot ||_2$, we see that the $2r$-th moment is finite.

The proof of part $a_2$) follows as in the above lemma.

To prove b), we apply the representation given in the proof of Theorem \ref{th:spARCH}; i.e.,
\begin{equation*}
\xvec{Y}^{(2)} = \sum_{v=0}^\infty \xmat{A}^{2v} \xvec{\eta} \, .
\end{equation*}
Now,
\begin{equation*}
|| \left(Y(\xvec{s}_1)^2, \ldots, Y(\xvec{s}_n)^2\right) || \; \leq \; \sum_{v=0}^{\infty} || \xmat{A}^2 ||^v
|| \xvec{\eta} || \; \leq \; \frac{1}{1-\lambda} ||\xvec{\eta}|| \, .
\end{equation*}
This completes the proof.
\end{proof}

\begin{proof}[Proof of Theorem \ref{th:triangularW}]
Because $\xmat{W}$ is a strictly triangular matrix, it follows that
\begin{equation*}
\det\left( \left( \frac{\partial y_j/\sqrt{h}_j}{\partial y_i} \right)_{i,j=1,\ldots,n}\right) = \frac{1}{\prod\limits_{j=1}^{n} \sqrt{h_j}} \, .
\end{equation*}
with $h_j = \alpha_j + \sum_{v=1}^{j-1} w_{jv} Y(\xvec{s}_v)^2$. Let $\xvec{Y}_k =(Y(\xvec{s}_1),\ldots,Y(\xvec{s}_k))^\prime$. Then,
\begin{equation*}
f_{\xvec{Y}_k}(\xvec{y}) = \prod\limits_{j=1}^{k} \frac{1}{\sqrt{h_j}} f_{\varepsilon(\xvec{s}_j)}\left( \frac{y_j}{\sqrt{h_j}}\right) \, .
\end{equation*}
Thus,
\begin{eqnarray*}
&& E(Y(\xvec{s}_k)^2 | Y(\xvec{s}_j), j = 1, \ldots, k-1) \\
& = &  \frac{1}{f_{Y(\xvec{s}_1),\ldots,Y(\xvec{s}_{k-1})}(y_1, \ldots, y_{k-1})}
\int\limits_{-\infty}^{\infty} y_k^2 \prod\limits_{j=1}^{k} \frac{1}{\sqrt{h_j}} f_{\varepsilon(\xvec{s}_j)} \left( \frac{y_j}{\sqrt{h_j}}\right) \, d \, y_k  \\
& = & \int\limits_{-\infty}^{\infty} y_k^2 \frac{1}{\sqrt{h_k}} f_{\varepsilon(\xvec{s}_k)} \left( \frac{y_k}{\sqrt{h_k}}\right) \, d \, y_k .
\end{eqnarray*}
Because it is assumed that $Var(\varepsilon(\xvec{s}_k)) = 1$ for all $k$, it follows that
\begin{equation*}
E(Y(\xvec{s}_k)^2 | Y(\xvec{s}_j), j = 1, \ldots, k-1) = h_k \, .
\end{equation*}
\end{proof}

\end{appendix}


\end{document}